\documentclass[
11pt,                          % standard font size
final,                         % draft or final?
english                        % standard language
]{article}

\usepackage{amsmath}           % ams mathematical stuff
\usepackage{amssymb}           % ams symbols
\usepackage[utf8]{inputenc}    % smart input of funny chars
\usepackage[T1]{fontenc}       % also for the font encoding
\usepackage[a4paper]{geometry} % geometry of page layout
\usepackage{setspace}          % for onehalfspace
\usepackage{bbm}               % nicer bbm fonts
\usepackage{mathrsfs}          % nice calligraphic font
\usepackage{mathtools}         % some nice tricks for stackrel and DeclarePairedDelimiter
\usepackage{stmaryrd}          % more brackets
\usepackage[amsmath,thmmarks,hyperref]{ntheorem} % nicer theorems
\usepackage[activate={true,nocompatibility},
            final,tracking=true,kerning=true,
            spacing=true,factor=1100,
            stretch=10,shrink=10,expansion=false
           ]{microtype}
\microtypecontext{spacing=nonfrench}  % microtype told me to do this...
\usepackage[hyperref]{xcolor}
\definecolor{mydarkblue}{RGB}{0,0,155}
\usepackage[%backref=page,      % backrefs in the bibliography
           final=true,         % always treat as final
           colorlinks=true,    % linkcolors are set to a dark blue
           allcolors=mydarkblue,  % this should turn out fine on a black and white printer
           hypertexnames=false,% better counter handling, avoids problems
           plainpages=false,              % whatever that means
           pdfpagelabels=true,            % strange options
           pdfencoding=auto,              % getting worse
           unicode=true                   % unicode is always nice
                      ]{hyperref}         % hyperrefs are cool!
\usepackage[shortcuts]{extdash}  % nonbreaking hyphens

\author{
  \textbf{Matthias Schötz}\thanks{Instytut Matematyczny PAN; ul. Śniadeckich 8, 00-656 Warsaw, Poland; \href{mailto:schotz@impan.pl}{schotz@impan.pl}}
}

\makeatletter

\renewcommand{\mathbb}[1]{\mathbbm{#1}}

\newcommand{\refitem}[1] {\textit{\ref{#1}.)}}

\theoremheaderfont{\normalfont\bfseries}
\theorembodyfont{\itshape}
\newtheorem{lemma}{Lemma}

\newtheorem{proposition}[lemma]{Proposition}
\newtheorem{theorem}[lemma]{Theorem}
\newtheorem{corollary}[lemma]{Corollary}
\newtheorem{definition}[lemma]{Definition}

\theorembodyfont{\rmfamily}

\newtheorem{example}[lemma]{Example}
\newtheorem{remark}[lemma]{Remark}

\theoremstyle{nonumberplain}
\theoremheaderfont{\normalfont\bfseries}
\theorembodyfont{\itshape}

\theoremheaderfont{\normalfont\bfseries}
\theorembodyfont{\normalfont}
\theoremseparator{:}
\theoremsymbol{\hbox{$\boxempty$}}
\newtheorem{proof}{Proof}

\newcommand{\I}              {\mathrm{i}}
\newcommand{\id}             {\mathrm{id}}
\newcommand{\argument}       {\ignorespaces{\,\cdot\,}\ignorespaces}

\DeclarePairedDelimiter{\abs}{\lvert}{\rvert}
\DeclarePairedDelimiter{\norm}{\lVert}{\rVert}

\DeclarePairedDelimiter{\ordinaryIP}{\langle}{\rangle}
\DeclarePairedDelimiter{\ordinarySet}{\{}{\}}
\newcommand{\QQ}{\mathbb{Q}}
\newcommand{\RR}{\mathbb{R}}
\newcommand{\FF}{\mathbb{F}}

\newcommand{\NN}{\mathbb{N}}
\newcommand{\Unit}{\mathbb{1}}
\newcommand{\Loc}{\mathrm{Loc}}
\newcommand{\ZZ}{\mathbb{Z}}
\newcommand{\CC}{\mathbb{C}}
\newcommand{\skal}[3][]{\ordinaryIP[#1]{\,#2 \,#1|\, #3\,}}
\newcommand{\set}[3][]{\ordinarySet[#1]{\,#2 \;#1|\; #3\,}}
\newcommand{\genRing}[2][]{#1\langle\!#1\langle\,#2\,#1\rangle\!#1\rangle_{\textup{rg}}}
\newcommand{\hermitian}{\mathrm{h}}
\newcommand{\ring}[1]{\mathcal{#1}}
\newcommand{\loc}{\mathrm{loc}}
\DeclareMathOperator{\dom}{dom}
\newcommand{\Hilb}{\mathscr{H}}
\newcommand{\Adbar}{\mathcal{L}^*}

\newcommand{\A}{\ring{A}}
\newcommand{\R}{\ring{R}}
\renewcommand{\S}{\ring{S}}
\newcommand{\at}[2][]{#1|_{#2}}
\newcommand{\Open}{\mathrm{O}}
\newcommand{\Dom}{\mathcal{D}}
\newcommand{\AlEv}{\mathscr{C}_{\mathrm{a.e.\!}}}
\newcommand{\Stetig}{\mathscr{C}}

\newcommand{\bd}{\mathrm{bd}}
\newcommand{\LocDomains}{\mathcal{D}_{\mathrm{loc}}}
\newcommand{\KSpace}{\mathcal{K}}
\newcommand{\KSpaceFin}{\mathcal{K}_{\mathrm{fin}}}
\newcommand{\monoid}[1]{\mathrm{#1}}
\newcommand{\group}[1]{\mathrm{#1}}
\newcommand{\widehatFin}[1]{\widehat{#1}_{\mathrm{fin}}}

\title{%
Rings of almost everywhere defined functions%
}

\date{June 2024}

\begin{document}
\begin{onehalfspace}

\maketitle
\begin{abstract}
  The following representation theorem is proven: A partially ordered commutative ring $\R$ is a subring of a ring of almost everywhere
  defined continuous real-valued functions on a compact Hausdorff space $X$ if and only if $\R$ is archimedean and localizable.
  Here we assume that the positive cone of $\R$ is closed under multiplication and stable under multiplication with squares,
  but actually one of these assumptions implies the other.
  An almost everywhere defined function on $X$ is one that is defined on a dense open subset of $X.$
  A partially ordered commutative ring $\R$ is archimedean if the underlying additive partially ordered abelian group is archimedean,
  and $\R$ is localizable essentially if its order is compatible with the construction of a localization with sufficiently large, positive denominators.
  As applications we discuss the $\sigma$-bounded case, lattice-ordered commutative rings ($f$\=/rings), partially ordered fields, and commutative operator algebras.
\end{abstract} \ \\[-0.5cm]
{\small
}

\allowdisplaybreaks

\section{Introduction}

Consider a commutative ring $\R,$ which we always assume to have a unit. For any field $\FF$ a question of general interest is whether or not $\R$ is isomorphic to
a subring of the ring $\FF^X$ of $\FF$\=/valued functions on a suitable set $X$ (with addition and multiplication defined pointwise).
For example, if $\FF$ is algebraically closed and $\R$ a finitely generated $\FF$-algebra (i.e.~a quotient of a polynomial ring over $\FF$), 
then Hilbert's Nullstellensatz \cite{hilbert:Nullstellensatz} provides the answer: $\R$ is isomorphic to a subring of $\FF^X$
for suitable $X$ if and only if $\R$ has no non-zero nilpotent elements.
If $\FF$ is the field of real numbers $\RR,$ then Hilbert's Nullstellensatz does not apply, but answers to the above question are provided e.g.~by the Gelfand--Naimark
theorem for commutative real $C^*$\=/algebras \cite[Lemma~1]{gelfand.naimark:ImbeddingofNormedRingsIntoRingOfOperators} and by the real Nullstellensatz \cite{risler:nullstellensatz}:
\begin{itemize}
  \item \label{item:GN} If $\R$ is a commutative real $C^*$\=/algebra with trivial $^*$\=/involution $r^* \coloneqq r$ for all $r\in \R,$ then $\R$ is isomorphic to the ring of $\RR$\=/valued continuous functions on a compact Hausdorff space $X.$
  \item \label{item:RN} If $\R$ is a finitely generated commutative real algebra and formally real (i.e.~$\sum_{j=1}^k r_j^2 = 0$ with $r_1, \dots, r_k \in \R$ implies 
    $r_1 = \dots = r_k = 0$),
    then $\R$ is isomorphic to the ring of polynomial functions on a real algebraic set.
\end{itemize}
However, it is not hard to show that there cannot be a general answer along the lines of the real Nullstellensatz --  
more precisely, whether or not $\R$ is isomorphic to
a subring of the ring $\RR^X$ on some suitable set $X$ cannot be decided by examining only finitely generated subrings of $\R$:

As an example take $\R = \RR(x),$ the field of rational functions in one variable $x.$ Every finitely generated
subring $\ring S \subseteq \RR(x)$ is isomorphic to a subring of $\RR^{\RR \setminus P}$ with $P\subseteq \RR$ the finite
set of poles of elements of $\ring S.$ Yet $\RR(x)$ itself cannot be isomorphic to any subring of
$\RR^X$ for any set $X$ because there does not even exist a single unital ring morphism $\varphi \colon \RR(x) \to \RR$
as this would lead to the contradiction $1 = \varphi\bigl((x-\varphi(x))^{-1}\bigr) \varphi\bigl(x-\varphi(x)\bigr) = 0.$

In contrast to this, a commutative ring $\R$ clearly is formally real if and only if every finitely generated subring of $\R$ is formally real.
To give just another elementary example, a ring is commutative if and only if all its finitely generated subrings are commutative.
Properties of this type are abundant, but they cannot characterize the rings of $\RR$-valued functions among all rings.

Note that rings of $\RR$\=/valued functions carry an additional structure, the partial order of pointwise comparison.
Because of this, questions of representability by $\RR$\=/valued functions should rather be examined for partially ordered commutative rings,
i.e.~for commutative rings endowed with a translation-invariant partial order $\le$ such that the positive cone $\R^+ \coloneqq \set{r\in \R}{0 \le r}$
is closed under multiplication and contains all squares. This, however, does not make a substantial difference:
whether or not a partially ordered commutative ring $\R$ is isomorphic to a subring of $\RR^X$ for some suitable set $X$
cannot be decided by examining only finitely generated subrings of $\R$ as demonstrated by the same example of $\RR(x),$ now also endowed with the partial
order of pointwise comparison almost everywhere. Any characterization of the rings of $\RR$-valued functions inside
all partially ordered commutative rings must therefore include a more involved condition which might feel less natural
(e.g.~existence of $\RR$-valued ring morphisms, or closedness of the positive cone in a topology with suitable properties,
see \cite[Cors.~3.4 and 4.23]{schoetz:gelfandNaimarkTheorems}, etc.).

The solution is to broaden the question and to consider representations
in the larger class of rings of almost everywhere defined continuous functions.
In this setting, the main Theorem~\ref{theorem:main} shows that for any partially ordered commutative ring $\R$ the following are equivalent:
\begin{itemize}
  \item There exists a compact Hausdorff space $X$ such that $\R$ is isomorphic (as a partially ordered commutative ring)
    to a subring of $\AlEv(X),$ the almost everywhere defined continuous $\RR$\=/valued functions on $X.$
  \item $\R$ is archimedean and strongly localizable.
\end{itemize}
In Proposition~\ref{proposition:fingenEnough} we also show that the properties of being archimedean or strongly localizable
can be checked by examining only finitely generated subrings.

A representation theorem of this type, i.e.~by almost everywhere defined continuous $\RR$\=/valued functions, has essentially already been
developed many decades ago for certain lattice-ordered rings, see \cite{henriksen.johnson:structureOfArchimedeanLatticeOrderedAlgebras}.
The formulation there is only slightly different, using continuous functions on a compact Hausdorff space $X$ with values
in the extended real line $\RR \cup \{-\infty,+\infty\}$ that are finite on an open dense subset of $X.$ This has the
disadvantage that such functions in general do not form a ring. Up to such technicalities, this classical result in the lattice-ordered
setting can essentially be obtained as a corollary of the main Theorem~\ref{theorem:main} here.
On the way we will also examine some weaker forms of representations, namely by equivalence classes of $\RR$\=/valued continuous functions
that coincide on some element of a given set of open (but not necessarily dense) subsets of a topological space.

This article is organized as follows: After discussing some preliminaries in the next Section~\ref{sec:preliminaries} we
define and examine the rings of almost everywhere defined continuous $\RR$\=/valued functions on a compact Hausdorff space in Section~\ref{sec:somewhereDefRings}.
In Sections~\ref{sec:characters} and \ref{sec:gelfand}
we introduce the main technical tools required in the proof of the main Theorem~\ref{theorem:main},
namely the space of extended positive characters of a partially ordered commutative ring
and the extended Gelfand transformation. Theorem~\ref{theorem:positivity} in Section~\ref{sec:gelfand},
characterizing the elements that have positive image under the extended Gelfand transformation, is also of independent interest.
In Section~\ref{sec:main} we then prove the main Theorem~\ref{theorem:main}, and in the final Section~\ref{sec:app} we discuss some applications:
The $\sigma$-bounded case, lattice-ordered rings ($f$-rings), partially ordered fields, and operator algebras (the first two of these
applications are known theorems, the last two seem to be new results).

\section{Preliminaries} \label{sec:preliminaries}

The natural numbers are denoted by $\NN \coloneqq \{1,2,3,\dots\}$ and $\NN_0 \coloneqq \NN \cup \{0\}.$ Similarly,
the rings of integers, rationals, real and complex numbers are $\ZZ,$ $\QQ,$ $\RR,$ and $\CC,$ respectively. 
We write $[x,y]$, ${[x,y[}$, ${]x,y]}$, and $]x,y[$ with $x,y\in \RR \cup \{-\infty,\infty\}$ for the closed, half-open, and open intervals, respectively.
A ring is always supposed to
have a multiplicative unit $1,$ and accordingly, a subring of a ring $\R$ is meant to contain the multiplicative unit of $\R.$
Next we recall some basic notions concerning partially ordered abelian groups and commutative rings:

A partially ordered abelian group is an abelian group $(\group G, +, 0)$ endowed with a translation-invariant partial order $\le,$
i.e.~$\le$ is a reflexive, transitive, and antisymmetric relation on $\group G$ and $g+f \le h+f$ holds for all $f,g,h\in \group G$ that fulfil $g\le h.$
In this case we write $\group G^+ \coloneqq \set{g\in \group G}{0 \le g}$ for the \emph{positive cone} of $\group G.$
This positive cone $\group G^+$ is a submonoid of $\group G$ such that $\group G^+ \cap (-\group G)^+ = \{0\},$ and two 
elements $g,h \in \group G$ fulfil $g \le h$ if and only if $h-g \in \group G^+.$
Conversely, given an abelian group $(\group G, +, 0),$ then any submonoid $M$ of $\group G$ satisfying $M \cap (-M) = \{0\}$
induces a unique translation-invariant partial order $\le$ on $\group G$ such that $\group G^+ = M.$ 

A partially ordered abelian group $(\group G, +, 0)$ is called \emph{archimedean} if the following holds:
Whenever two elements $g,h \in \group G$ fulfil $kg + h \in \group G^+$ for all $k\in \NN,$ then $g \in \group G^+.$
This archimedean property is crucial in the theory of partially ordered abelian groups, see e.g.~\cite{goodearl:partiallyOrderedAbelianGroupsWithInterpolation}.
We will call a commutative ring endowed with a translation-invariant partial order archimedean if the underlying partially ordered abelian group is archimedean.
This should not be confused with the unrelated notion of archimedean preorderings (preprimes / semirings, quadratic modules, ...)
of commutative rings in real algebraic geometry. For any submonoid $\monoid M$ of an abelian group $(\group G,+,0)$ we write
\begin{equation}
  \label{eq:ddagger}
  M^\ddagger \coloneqq \set[\big]{g\in \group{G}}{\textup{there is $h\in \group{G}$ such that $k g+h \in \monoid M$ for all $k \in \NN$}}
  .
\end{equation}
It is easy to check that $M^\ddagger$ is again a submonoid of $\group{G}$ and $M \subseteq M^\ddagger.$
The operation $\argument^\ddagger$ has been examined in various contexts before, mostly for convex cones of real vector
spaces where $\argument^\ddagger$ is a sequential closure as in 
\cite{cimpric.marshall.netzer:closuresOfQuadraticModules, kuhlmann.marshall:positivitySOSandMultidimensionalMomentProblem, netzer:positivePolynomialsAndSequentialClosuresOfQuadraticModules, schmuedgen.schoetz:positivstellensaetzForSemirings}.
In \cite{cimpric.marshall.netzer:closuresOfQuadraticModules} one also finds a counterexample showing that in general
$({\monoid M}^\ddagger)^\ddagger \supsetneq {\monoid M}^\ddagger.$ A partially ordered abelian group $\group{G}$ is archimedean
if and only if $(\group{G}^+)^\ddagger = \group{G}^+.$

Let $\R$ be a commutative ring, then a \emph{preordering} of $\R$ is a subset
$T$ of $\R$ that is closed under addition and multiplication and such that $r^2 \in T$ for all $r\in \R.$
A \emph{partially ordered commutative ring} is a commutative ring $\R$ endowed with a translation-invariant order $\le$ on the underlying additive group
such that $\R^+$ is a preordering of $\R.$ In this case $rt \le st$ holds for all $r,s\in \ring R$ that fulfil $r\le s$ and all $t \in \R^+.$
All non-trivial partially ordered commutative rings fulfil $1 \in \R^+$ but $-1 \notin \R^+$, and in particular they have characteristic $0.$
Any subring $\S$ of a partially ordered commutative ring $\R$ automatically becomes a partially ordered commutative ring by endowing $\S$ with
the relative order inherited from $\R.$

Now consider two partially ordered commutative rings $\R$ and $\S$ and a ring morphism $\Phi \colon \R \to \S.$ We say that
$\Phi$ is \emph{positive} if $\Phi(r) \in \S^+$ for all $r\in \R^+,$ i.e.~$\Phi(\R^+) \subseteq \S^+.$ 
A positive ring morphism $\Phi \colon \R \to \S$ is called an \emph{order embedding} if conversely also $\Phi^{-1}(\S^+) \subseteq \R^+$
holds (hence $\Phi^{-1}(\S^+) = \R^+$ by positivity of $\Phi$). In this case $\Phi$ is automatically injective.
An isomorphism of the partially ordered commutative rings $\R$ and $\S$
is an isomorphism $\Phi$ of the underlying rings such that $\Phi$ and its inverse $\Phi^{-1}$ are positive.
Equivalently, a positive ring morphism $\Phi \colon \R \to \S$ is an isomorphism of partially ordered commutative rings
if and only if $\Phi$ is a surjective order embedding.

We are interested in partially ordered commutative rings that fulfil an additional assumption:

\begin{definition} \label{def:localizability}
  Let $\R$ be a commutative ring endowed with a translation-invariant partial order $\le$ on the underlying additive group.
  An element $s\in 1+\R^+ \coloneqq \set{1+t}{t\in \R^+}$ is called \emph{localizable} if the following holds:
  Whenever an element $r\in \R$ fulfils $rs \in \R^+,$ then $r\in \R^+.$ We write $\Loc(\R)$ for the set of all localizable elements in $1+\R^+$
  and we say that $\R$ itself is \emph{localizable} if for all $r\in \R$ there exists $s\in \Loc(\R)$ such that $-s \le r \le s.$
  Moreover, we say that $\R$ is \emph{strongly localizable} if $r^2 \in \R^+$ for all $r\in \R$ and $\Loc(\R) = 1+\R^+.$
\end{definition}
Clearly $1 \in \Loc(\R)$ and $\Loc(\R)$ is closed under multiplication.
We also make the following observation: If $\R$ is localizable, then $\Loc(\R) \subseteq \R^+.$ Indeed, if $\R$ is localizable,
then there exists at least one element $s\in \Loc(\R)$ that fulfils $-s \le 0 \le s,$ in particular $s\in \R^+,$
therefore $1\in \R^+$ by localizability of $s,$ and consequently $\Loc(\R) \subseteq 1+\R^+ \subseteq \R^+.$

Localizability of elements of rings with a partial order has been introduced in \cite{schmuedgen.schoetz:positivstellensaetzForSemirings}
but has implicitly been present in real algebraic geometry for a long time, e.g.~in the denominator of the strict Positivstellensatz.
In Section~\ref{sec:characters} we will see that the localization $\ring{R}_\loc$ of a partially ordered commutative ring $\ring{R}$ over $\Loc(\R)$
can again be turned into a partially ordered commutative ring, and if $\ring{R}$ is localizable, then
$\ring{R}_\loc$ has many uniformly bounded elements. Localizability is also the key to multiplicativity of extremal
positive functionals, see \cite[Thm.~4.20]{schoetz:gelfandNaimarkTheorems}, and to automatic associativity and commutativity
of the multiplication of partially ordered (extended) rings, see \cite{schoetz:arxivAssociativityAndCommutativityOfPartiallyOrderedRings}.
As strong localizability requires squares to be positive, it indeed implies localizability as the name suggests:

\begin{proposition} \label{proposition:weakstronloc}
  Let $\R$ be a commutative ring endowed with a translation-invariant partial order $\le$ on the underlying additive group
  and assume that $\R$ is strongly localizable, then $\R$ is localizable.
\end{proposition}
\begin{proof}
  Consider an element $r\in \R.$ Then $2\bigl( (1+r)^2 - r \bigr) = (1+r)^2 + 1^2 + r^2 \in \R^+$ holds, and therefore
  $(1+r)^2 - r \in \R^+$ because $2 = 1+1^2 \in 1+\R^+ \subseteq \Loc(\R).$
  This shows that $r\le (1+r)^2,$ and, by substituting $r$ with $-r,$ $-r \le (1-r)^2.$
  So set $s \coloneqq (1+r)^2 + (1-r)^2 = 1+1^2+2r^2 \in 1+\R^+,$ then $-s \le r \le s$ and $s \in \Loc(\R)$ by strong localizability.
\end{proof}
In the archimedean case, the converse is often also true as we will see in the main Theorem~\ref{theorem:main}.
This has also already been observed in \cite[Thm.~32]{schoetz:arxivAssociativityAndCommutativityOfPartiallyOrderedRings}.

Definition~\ref{def:localizability} was formulated in a setting more general than partially ordered commutative rings 
because of the following two technical propositions inspired by \cite[Thm.~3.2 and Prop.~3.5]{schmuedgen.schoetz:positivstellensaetzForSemirings}.
These show that some axioms of localizable archimedean partially ordered commutative rings are redundant.
This is not necessary for the main Theorem~\ref{theorem:main}, but it will later on be helpful
for the construction of examples in the last section~\ref{sec:app}. As a preparation we note:

\begin{lemma} \label{lemma:archLoc}
  Let $\R$ be a commutative ring endowed with a translation-invariant partial order $\le$ on the underlying additive group.
  Assume also that $\R$ is archimedean and localizable. Then $\NN \subseteq \Loc(\R).$
\end{lemma}
\begin{proof}
  Recall that $1 \in \Loc(\R) \subseteq \R^+$ by localizability and therefore $\NN \subseteq 1+\R^+.$
  Consider $n\in \NN$ and let $r\in \R$ be given such that $nr \in \R^+.$ For all $m\in \{1,\dots,n-1\}$ there exists
  $s_m \in \Loc(\R) \subseteq \R^+$ such that $-s_m \le mr \le s_m.$ Set $\hat s \coloneqq \sum_{m=1}^{n-1} s_m,$ then $0 \le mr + s_m \le mr + \hat s$
  for all $m\in \{1,\dots,n-1\}.$ It follows that $\ell nr + mr + \hat s \in \R^+$
  for all $\ell \in \NN_0$ and all $m\in \{1,\dots,n-1\},$ so $kr + \hat s \in \R^+$ for all $k\in \NN,$
  and therefore $r \in \R^+$ because $\R$ is archimedean.
\end{proof}

\begin{proposition} \label{proposition:semiring}
  Let $\R$ be a commutative ring endowed with a translation-invariant partial order $\le$ on the underlying additive group
  such that $\R^+$ is closed under multiplication. If $\R$ is archimedean and localizable, then $r^2 \in \R^+$ for all 
  $r\in \R$ so that $\R$ in particular is a partially ordered commutative ring.
\end{proposition}
\begin{proof}
  Assume that $\R$ is archimedean and localizable, then $\NN \subseteq \Loc(\R)$ by the previous Lemma~\ref{lemma:archLoc}.
  Let any $r\in \R$ be given, then there is $s\in \Loc(\R)$ such that $-s \le r \le s.$ For $k\in \NN \setminus \{1\}$ define
  \begin{equation*}
    S_k \coloneqq \sum_{\ell=0}^k \binom{k}{\ell} (k-2\ell)^2 (s+r)^{k-\ell} (s-r)^\ell,
  \end{equation*}
  then $S_k \in \R^+$ because $s+r,s-r \in \R^+$ and because $\R^+$ is closed under multiplication by assumption. 
  A short calculation analogous to \cite[Lemma~3.1]{schmuedgen.schoetz:positivstellensaetzForSemirings} shows that
  $S_k = 4 k (2s)^{k-2} s^2 + 4 k (k-1) (2s)^{k-2} r^2$. For convenience of the reader, we give the details here:
  Expanding $(k-2\ell)^2 = k^2 - 4 \ell (k-1) + 4\ell(\ell-1)$ and applying the binomial theorem shows that
  \begin{align*}
    S_k
    &=
    k^2 (2s)^k  - 4 (k-1) \sum_{\ell=0}^k \binom{k}{\ell} \ell (s+r)^{k-\ell} (s-r)^\ell + 4 \sum_{\ell=0}^k \binom{k}{\ell} \ell(\ell-1) (s+r)^{k-\ell} (s-r)^\ell
    \\
    &=
    k^2 (2s)^k  - 4 k (k-1) \sum_{\ell=1}^k \binom{k-1}{\ell-1} (s+r)^{k-\ell} (s-r)^\ell + 4 k(k-1) \sum_{\ell=2}^k \binom{k-2}{\ell-2} (s+r)^{k-\ell} (s-r)^\ell
    \\
    &=
    k^2 (2s)^k - 4 k (k-1) \sum_{\ell=0}^{k-1} \binom{k}{\ell} (s+r)^{k -1 - \ell} (s-r)^{\ell+1} + 4 k(k-1) \sum_{\ell=0}^{k-2} \binom{k}{\ell} (s+r)^{k-2-\ell} (s-r)^{\ell+2}
    \\
    &=
    k^2 (2s)^k - 4 k (k-1) (2s)^{k-1} (s-r) + 4 k(k-1) (2s)^{k-2} (s-r)^2
    \\
    &=
    k^2 (2s)^k - 4 k (k-1) (2s)^{k-2} \bigl( s^2 - r^2 \bigr)
    \\
    &=
    4 k (2s)^{k-2} s^2 + 4 k (k-1) (2s)^{k-2} r^2
    .
  \end{align*}
  It follows that $s^2 + (k-1) r^2 \in \R^+$ because $\NN \subseteq \Loc(\R)$ and $s\in \Loc(\R).$
  This shows that $s^2 + k r^2 \in \R^+$ for all $k\in \NN$ and therefore $r^2 \in \R^+$ because $\R$ is archimedean.
\end{proof}

\begin{proposition} \label{proposition:sqrt}
  Let $\R$ be a commutative ring endowed with a translation-invariant partial order $\le$ on the underlying additive group
  such that $\R^+$ is stable under multiplication with squares, i.e.~$r^2s \in \R^+$ for all $r\in \R,$ $s\in \R^+.$
  If $\R$ is archimedean and localizable, then $pq\in \R^+$ for all $p,q\in \R^+$ so that $\R$ in particular is a partially ordered commutative ring.
\end{proposition}
\begin{proof}
  Assume that $\R$ is archimedean and localizable, then $\NN \subseteq \Loc(\R)$ by Lemma~\ref{lemma:archLoc}.
  Recall also that $\Loc(\R) \subseteq \R^+$ so that $1 \in \R^+$ and therefore $r^2 \in \R^+$ for all $r\in \R$ by assumption.
  Let $p,q\in \R^+$ be given. In order to show that $pq \in \R^+$ we essentially approximate a square root of $p.$
  By localizability of $\R$ there exists $s\in \Loc(\R)$ such that $p \le s.$ We will use the shorthand
  $s_k \coloneqq s^{(2^k)},$ so $s_0 = s,$ and $s_{k+1} = s_k^2$ and $s_k \in \Loc(\R)$ hold for all $k\in \NN_0.$
  The crucial technical step is to show that $2^{\ell+1} pq + s(1+q^2) \in \R^+$ for all $\ell \in \NN.$
  
  We recursively define a sequence $(r_k)_{k\in \NN_0}$ in $\R$ as $r_0 \coloneqq p$ and $r_{k+1} \coloneqq 2 s_k r_k - r_k^2$ for all $k\in \NN_0.$
  Then $s_{k+1} - r_{k+1} = (s_k - r_k)^2 \in \R^+$ for all $k\in \NN_0,$ and therefore $s_k-r_k \in \R^+$ for all $k\in \NN_0$
  because in the case $k=0$ this reduces to the assumption that $p \le s.$ Moreover, we can prove by induction that $r_k \in \R^+$
  for all $k\in \NN_0$: For $k=0$ this is just the assumption that $p\in \R^+$; if $r_k \in \R^+$ for some $k\in \NN_0,$
  then $r_k^2(s_k-r_k) + (s_k-r_k)^2 r_k + s_k^2 r_k \in \R^+$ because $\R^+$ is stable under multiplication with squares, and a quick calculation
  shows that $r_k^2(s_k-r_k) + (s_k-r_k)^2 r_k + s_k^2 r_k = s_k r_{k+1},$ so $r_{k+1} \in \R^+$ by localizability of $s_k.$
  We have thus shown that $0 \le r_k \le s_k$ for all $k\in \NN_0.$
  For any $\ell \in \NN$ define
  \begin{equation*}
    S_\ell \coloneqq \sum_{k=0}^{\ell-1} 2^{\ell-k} \Bigl( \prod\nolimits_{j=k}^{\ell-1} s_j^2 \Bigr) r_k^2 q,
  \end{equation*}
  then $S_\ell \in \R^+$ because $q\in \R^+$ and because $\R^+$ is stable under multiplication with squares.
  Moreover,
  \begin{align*}
    S_\ell
    &=
    \sum_{k=0}^{\ell-1} 2^{\ell-k} \Bigl( \prod\nolimits_{j=k+1}^\ell s_j \Bigr) r_k^2 q
    \\
    &=
    \sum_{k=0}^{\ell-1} 2^{\ell-k} \Bigl( \prod\nolimits_{j=k+1}^\ell s_j \Bigr) \bigl(2 s_k r_k - r_{k+1}\bigr) q
    \\
    &=
    \sum_{k=0}^{\ell-1} 2^{(\ell+1)-k} \Bigl( \prod\nolimits_{j=k}^\ell s_j \Bigr) r_k q
    -
    \sum_{k=0}^{\ell-1} 2^{(\ell+1)-(k+1)} \Bigl( \prod\nolimits_{j=k+1}^\ell s_j \Bigr) r_{k+1} q
    \\
    &=
    2^{\ell+1} \Bigl( \prod\nolimits_{j=0}^\ell s_j \Bigr) p q - 2 s_\ell r_\ell q
  \end{align*}
  and therefore $2^{\ell+1} \bigl( \prod\nolimits_{j=0}^{\ell-1} s_j \bigr) p q - 2 r_\ell q \in \R^+$ by localizability of $s_\ell.$
  Now note that
  \begin{equation*}
    s_\ell (1+q^2)  + 2 r_\ell q = (s_\ell - r_\ell)(1+q^2) + r_\ell (1+q^2) + 2 r_\ell q = (s_\ell - r_\ell)(1^2+q^2) + r_\ell (1+q)^2 \in \R^+
  \end{equation*}
  because $s_\ell - r_\ell \in \R^+$ and $r_\ell \in \R^+$ as discussed above and because $\R^+$ is stable under multiplication with squares.
  Adding these two elements of $\R^+$ yields
  $2^{\ell+1} \bigl( \prod\nolimits_{j=0}^{\ell-1} s_j \bigr) p q + s_\ell (1+q^2) \in \R^+.$
  Finally, note that $s_\ell = \bigl( \prod\nolimits_{j=0}^{\ell-1} s_j \bigr) s,$ which is easily checked by induction over $\ell,$
  so by localizability of $s_j$ with $j\in \{0,\dots,\ell-1\}$ we obtain the desired result that $2^{\ell+1} p q + s (1+q^2) \in \R^+.$
  
  In order to complete the proof it only remains to show that this implies $k pq + s(1+q^2) \in \R^+$ for all $k\in \NN,$ then $pq \in \R^+$
  because $\R$ is archimedean by assumption. So let $k\in \NN$ be given, then there exists $\ell \in \NN$ such that $k\le 2^{\ell+1}$
  and $2^{\ell+1} k p q + k s (1+q^2) \in \R^+.$ Note that $s(1+q^2) \in \R^+$ because $s\in \R^+$ and because $\R^+$ is stable under multiplication with squares.
  Therefore $(2^{\ell+1} - k) s (1+q^2) \in \R^+$ and consequently $2^{\ell+1} k p q + 2^{\ell+1} s (1+q^2) \in \R^+.$
  As $2 \in \Loc(\R)$ this shows that $kpq + s(1+q^2) \in \R^+.$
\end{proof}

We also note that strong localizability and the archimedean property descend to subrings, i.e.~if a partially ordered commutative
ring $\R$ is archimedean or strongly localizable, then all its subrings are also archimedean or strongly localizable, respectively.
Conversely, strong localizability and the archimedean property can actually be checked by looking at only the finitely generated subrings:

\begin{proposition} \label{proposition:fingenEnough}
  Let $\R$ be a partially ordered commutative ring, then the following holds:
  \begin{enumerate}
    \item $\R$ is archimedean if (and only if) all finitely generated subrings of $\R$ are archimedean.
    \item $\R$ is strongly localizable if (and only if) all finitely generated subrings of $\R$ are strongly localizable.
  \end{enumerate}
\end{proposition}

\begin{proof}
  The ``only if''\,-part is clear in both cases. Given $r,s\in \R$ then we write $\genRing{\{r,s\}}$
  for the subring (with unit) of $\R$ that is generated by the two elements $r$ and $s.$
  
  First assume all finitely generated subrings of $\R$ are archimedean. If $r,s\in \R$ fulfil $k r + s \in \R^+$ for all $k \in \NN,$
  then $k r + s \in \bigl(\genRing{\{r,s\}}\bigr)^+$ for all $k\in \NN$ and therefore $r \in \bigl(\genRing{\{r,s\}}\bigr)^+\subseteq \R^+$
  because $\genRing{\{r,s\}}$ is archimedean. This shows that whole $\R$ is archimedean.

  Now assume all finitely generated subrings of $\R$ are strongly localizable. By the definition of partially ordered commutative rings,
  $r^2 \in \R^+$ for all $r\in \R.$ If $r\in \R,$ $s\in 1+\R^+$ fulfil $rs \in \R^+,$
  then $s \in 1+\bigl(\genRing{\{r,s\}}\bigr)^+$ and $rs \in \bigl(\genRing{\{r,s\}}\bigr)^+,$ and consequently 
  $r \in \bigl(\genRing{\{r,s\}}\bigr)^+\subseteq \R^+$ by strong localizability of $\genRing{\{r,s\}}.$ This shows that whole $\R$ is strongly localizable.
\end{proof}

\section{Almost everywhere defined continuous functions} \label{sec:somewhereDefRings}

Consider any map $\Phi \colon X \to Y$ between two sets $X$ and $Y,$ then we write ${\dom \Phi} \coloneqq X$ for its \emph{domain}.
We also write $\Phi\at{A} \colon A \to Y,$ $x\mapsto \Phi\at{A}(x) \coloneqq \Phi(x)$ for the \emph{restriction} of $\Phi$
to a subset $A$ of $X.$

For the partially ordered commutative ring of continuous $\RR$\=/valued functions (with pointwise operations and pointwise order) on a topological space $X$ we write
$\Stetig(X) \coloneqq \set{f\colon X \to \RR}{f\textup{ continuous}}.$ We also write $\Stetig(A)$ for the ring of
continuous $\RR$\=/valued functions defined on a subset $A$ of the topological space $X,$ where it is understood that $A$ carries
the relative topology of $X.$

Recall that a \emph{$\pi$-system} $\Dom$ on a set $X$ is a non-empty set of subsets of $X$ which is closed under finite intersections,
i.e.~$A\cap B \in \Dom$ for all $A,B\in \Dom$. In particular the topology of any topological space $X$ is a $\pi$-system on $X$.
Accordingly, if $X$ is a topological space, then a \emph{$\pi$-subsystem} $\Dom$ of the topology of $X$ is a non-empty set
of open subsets of $X$ which is closed under finite intersections.

\begin{definition}
  For any $\pi$-subsystem $\Dom$ of the topology of a topology space $X$ we define addition and multiplication on the (disjoint) union $\bigcup_{A\in \Dom} \Stetig(A)$:
  Given $f,g \in \bigcup_{A\in \Dom} \Stetig(A),$ then their sum and product $f+g, fg \in \Stetig({{\dom f}} \cap {{\dom g}}) \subseteq \bigcup_{A\in \Dom} \Stetig(A)$
  are defined as the pointwise sum and products of their restrictions to ${\dom f} \cap {\dom g},$ i.e.
  \begin{equation}
    f+g \coloneqq f \at{{\dom f}\,\cap\,{\dom g}} + g\at{{{\dom f}}\,\cap\,{\dom g}}
    \quad\quad\textup{and}\quad\quad
    fg \coloneqq f \at{{{\dom f}}\,\cap\,{\dom g}} \, g\at{{{\dom f}}\,\cap\,{\dom g}}
    .
  \end{equation}
\end{definition}
One easily checks that these operations are associative and commutative and they fulfil the usual distributivity of multiplication over addition.
Moreover, consider $B \in \Dom$ and $0_B, \Unit_B \in \Stetig(B)$ the constant $0$- and $1$-functions, then $f + 0_B = f\at{{\dom f} \cap B} = f \Unit_B$
for all $f \in \bigcup_{A\in \Dom} \Stetig(A).$ By passing to a suitable quotient we will obtain a partially ordered commutative ring:

\begin{definition}
  Consider a $\pi$-subsystem $\Dom$ of the topology of a topological space $X$.
  We define relations $\lesssim$ and $\approx$ on $\bigcup_{A\in \Dom} \Stetig(A)$ as follows: Given $f,g \in \bigcup_{A\in \Dom} \Stetig(A),$
  then $f\lesssim g$ if and only if there exists $A \in \Dom,$ $A \subseteq {{\dom f}} \cap {\dom g},$ such that $f\at{A} \le g\at{A}$ holds
  with respect to the order of $\Stetig(A),$ i.e.~pointwise.
  Moreover, $f\approx g$ if and only if $f\lesssim g$ and $g\lesssim f.$
\end{definition}
One easily checks that $\lesssim$ is a quasi-order on $\bigcup_{A\in \Dom} \Stetig(A),$ i.e.~a reflexive and transitive
relation. This quasi-order is also compatible with addition and multiplication in the sense that
\begin{equation}
  c+e \lesssim d+e
  ,\quad\quad
  0_B \lesssim fg
  ,\quad\quad\textup{and}\quad\quad
  0_B \lesssim e^2
  \label{eq:ringorderDom}
\end{equation}
hold for all $c,d,e,f,g \in \bigcup_{A\in \Dom} \Stetig(A)$ and $B \in \Dom$ with $c\lesssim d$ and $0_B \lesssim f,$ $0_B \lesssim g.$
Moreover, $\approx$ is an equivalence relation on $\bigcup_{A\in \Dom} \Stetig(A)$ that is explicitly given,
for $f,g\in \bigcup_{A\in \Dom} \Stetig(A)$, by $f \approx g$ if and only if there exists $A\in \Dom$, $A \subseteq {\dom f} \cap {\dom g}$
such that $f\at A = g\at A.$ In particular $0_A \approx 0_B$ and $\Unit_A \approx \Unit_B$ for all $A,B\in \Dom$.
It is now easy to check that addition descends to the quotient $\bigl(\bigcup_{A\in \Dom} \Stetig(A)\bigr) / {\approx}$.
Moreover, as $f+(-f) = 0_{\dom f}$ for all $f\in\bigl(\bigcup_{A\in \Dom} \Stetig(A)\bigr) / {\approx}$,
this quotient $\bigl(\bigcup_{A\in \Dom} \Stetig(A)\bigr)  / {\approx}$ becomes an abelian group. Similarly,
multiplication also descends to the quotient $\bigl(\bigcup_{A\in \Dom} \Stetig(A)\bigr) / {\approx}$, turning it into
a commutative ring.
Finally, if $f,f',g,g' \in \bigcup_{A\in \Dom} \Stetig(A)$ fulfil $f \approx f',$ $g\approx g',$ and $f\lesssim g,$ then also $f' \lesssim g'$
by transitivity of ${\lesssim}.$ This allows to equip the quotient $\bigl(\bigcup_{A\in \Dom} \Stetig(A)\bigr) / {\approx}$
with a well-defined order $\le$:

\begin{definition}
  Consider a $\pi$-subsystem $\Dom$ of the topology of a topological space $X$.
  Then we define the commutative ring $\Stetig_\approx(\Dom) \coloneqq \bigl(\bigcup_{A\in \Dom} \Stetig(A)\bigr)  / {\approx}$
  and the canonical projection onto the quotient $[\argument] \colon \bigcup_{A\in \Dom} \Stetig(A) \to \Stetig_\approx(\Dom).$
  Moreover, we endow $\Stetig_\approx(\Dom)$ with the relation $\le$ that is defined, for $f,g\in \bigcup_{A\in \Dom} \Stetig(A),$
  as $[f] \le [g]$ if and only if $f \lesssim g.$
\end{definition}
From \eqref{eq:ringorderDom} it follows that $\Stetig_\approx(\Dom)$ is a partially ordered commutative ring. Moreover:

\begin{proposition} \label{proposition:functionsLocalizable}
  Consider a $\pi$-subsystem $\Dom$ of the topology of a topological space $X$.
  Then the partially ordered commutative ring $\Stetig_\approx(\Dom)$ is strongly localizable.
\end{proposition}
\begin{proof}
  Assume $f,g\in \bigcup_{A\in \Dom} \Stetig(A)$ fulfil $[g] \in \Stetig_\approx(\Dom)^+$ and $[f(\Unit+g)] \in \Stetig_\approx(\Dom)^+.$
  This means that there are $A,B \in \Dom$ such that $g(x) \ge 0$ for all $x\in A$ and $f(x)\bigl(1+g(x)\bigr) \ge 0$ for all $x\in B.$
  Consequently $f(x) \ge 0$ for all $x\in A \cap B,$ so $[f] \in \Stetig_\approx(\Dom)^+.$
\end{proof}
In the construction of such strongly localizable partially ordered commutative rings $\Stetig_\approx(\Dom)$ it is of course desirable
that the empty set $\emptyset$ is not an element of $\Dom$: 
The ring $\Stetig_\approx(\Dom)$ is trivial, i.e.~$[0] = [\Unit],$ if and only if $\emptyset \in \Dom.$

\begin{example}
  Consider a $\pi$-subsystem $\Dom$ of the topology of a topological space $X$
  and assume that $\bigcap_{A\in \Dom} A \in \Dom$ (this is the case e.g.\ if $\Dom$ consists of only finitely many elements).
  Set $A_{\min} \coloneqq \bigcap_{A\in \Dom} A.$
  Then the relation $\lesssim$ on $\bigcup_{A\in \Dom} \Stetig(A)$ can be described as
  $f\lesssim g$ for $f,g\in \bigcup_{A\in \Dom}\Stetig(A)$ if and only if $f\at{A_{\min}} \le g\at{A_{\min}}.$
  It is now easy to check that the map $\Stetig(A_{\min}) \ni f \mapsto [f] \in \Stetig_\approx(\Dom)$ is an isomorphism
  of partially ordered rings.
\end{example}

\begin{example}
  Let $X$ be a topological space, $C$ a non-empty closed subset of $X,$ and $\Dom$ the set of all open subsets of $X$ that contain $C.$
  Then $\Dom$ is a $\pi$-subsystem of the topology of $X$ and $\Stetig_\approx(\Dom)$ is the usual partially ordered commutative ring of germs
  of continuous $\RR$\=/valued functions around $C.$
  Examples of this type are typically not archimedean, e.g.~consider $X \coloneqq [-1,1]$ and $C \coloneqq \{0\}.$
  Then the germ at $0$ of $f \in \Stetig([-1,1]),$ $x \mapsto f(x) \coloneqq -x^2$ is not positive,
  but the germ at $0$ of $1+kf$ is positive for all $k\in \NN.$
\end{example}

\begin{lemma} \label{lemma:AlEv}
  Let $X$ be a topological space and let $\Dom$ be the set of dense open subsets of $X.$ Then $\Dom$ is a $\pi$-subsystem of the topology of $X$.
\end{lemma}
\begin{proof}
  $\Dom$ is non-empty because $X\in \Dom$. 
  Consider $A,B\in \Dom.$ Given any non-empty open subset $C$ of $X,$ then the open subset
  $A\cap C$ of $X$ is non-empty because $A$ is dense in $X,$ therefore $(A \cap B) \cap C = B\cap(A\cap C)$ is non-empty because $B$ is dense in $X.$
  This shows that the open subset $A\cap B$ of $X$ is again dense in $X,$ i.e.~$A\cap B \in \Dom.$
\end{proof}
Due to this Lemma~\ref{lemma:AlEv} we can define a certain partially ordered commutative ring $\AlEv(X)$ on any topological space $X$:

\begin{definition} \label{definition:AlEv}
  Let $X$ be a topological space, then we write $\AlEv(X) \coloneqq \Stetig_\approx(\Dom)$ with $\Dom$ the set of dense open subsets of $X$.
  An element of $\AlEv(X)$ is called an \emph{almost everywhere defined continuous $\RR$\=/valued function on $X$}.
\end{definition}

\begin{proposition} \label{proposition:alEvArchLoc}
  Let $X$ be a compact Hausdorff space, then $\AlEv(X)$ and all its subrings are strongly localizable archimedean partially ordered commutative rings.
\end{proposition}
\begin{proof}
  By Proposition~\ref{proposition:functionsLocalizable}, $\AlEv(X)$ is a strongly localizable partially ordered commutative ring,
  so all its subrings are strongly localizable, too.
  It only remains to check that $\AlEv(X)$ is archimedean (which also implies that all its subrings are archimedean):
  
  Let $\Dom$ be the set of dense open subsets of $X$ and $f,g \in \bigcap_{A\in \Dom} \Stetig(A)$
  such that $[kf+g] \in \AlEv(X)^+$ for all $k\in \NN.$ This means there is a sequence
  $(A_k)_{k\in \NN}$ in $\Dom$ such that $A_k \subseteq {{\dom f}} \cap {\dom g}$ and $(kf+g)\at{A_k} \ge 0$ pointwise
  for all $k\in \NN.$ Write $A_\infty \coloneqq \bigcap_{k\in \NN} A_k$ and consider any $x\in A_\infty,$ then the estimate
  $kf(x) + g(x) \ge 0$ holds for all $k\in \NN.$ This shows that $f(x) \ge 0$ for all $x\in A_\infty.$
  As $X$ is a compact Hausdorff space by assumption, and in particular a Baire space, $A_\infty$ is still a dense subset of $X$ and
  therefore $A_\infty$ is dense in ${{\dom f}}.$ It follows that $f(x) \ge 0$ for all $x\in {{\dom f}}$ by 
  continuity of $f$, hence $[f] \in \AlEv(X)^+.$
\end{proof}
Our main Theorem~\ref{theorem:main} essentially shows that the converse of this Proposition~\ref{proposition:alEvArchLoc} is also true.

\begin{proposition} \label{proposition:maxDomain}
  Let $X$ be a topological space and $a \in \AlEv(X),$ then there is a unique representative $a_{\max} \in a$
  that fulfils $f = a_{\max} \at{{{\dom f}}}$ for all $f\in a.$ In particular, $\dom a_{\max} \supseteq {{\dom f}}$ for all $f\in a.$
\end{proposition}

\begin{proof}
  The union $\bigcup_{f \in a} {{\dom f}}$ of open subsets of $X$ is again an open subset of $X.$
  As the equivalence class $a$ is non-empty, there exists $f\in a,$ and as ${{\dom f}}$ is already dense in $X$ it follows that 
  $\bigcup_{f \in a} {{\dom f}}$ is also dense in $X.$ Now consider $x\in \bigcup_{f \in a} {{\dom f}}$ and $f,f' \in a$ such that 
  $x\in {{\dom f}} \cap {\dom f'}.$ There exists a dense open subset $A$ of $X$ fulfilling $A \subseteq {\dom f} \cap {\dom f'}$ and $f\at A = f'\at A.$
  Then $A$ is also dense in ${\dom f} \cap {\dom f'},$ so $f\at{{\dom f} \cap {\dom f'}} = f'\at{{\dom f} \cap {\dom f'}}$
  by continuity of $f$ and $f',$ and in particular $f(x) = f'(x).$ It follows that the function
  $a_{\max} \colon \bigcup_{f \in a} {\dom f} \to \RR,$
  \begin{equation}
    x \mapsto a_{\max}(x) \coloneqq f(x)\quad\textup{with $f \in a$ any representative for which $x\in{\dom f}$}
  \end{equation}
  is well-defined. By this definition, $a_{\max}(x) = f(x)$ for all $f\in a$ and all $x\in {\dom f},$ i.e.~$f = a_{\max} \at{{\dom f}}$ for all $f\in a.$
  Consequently $a_{\max}$ is continuous, because for all $x\in \bigcup_{f \in a} {\dom f}$ there exists $f\in a$ with $x\in {\dom f},$
  so that ${\dom f}$ is an open neighbourhood of $x$ on which $a_{\max}$ coincides with the continuous function $f.$ As the domain of $a_{\max},$
  i.e.~$\bigcup_{f \in a} {\dom f},$ is a dense open subset of $X,$ it follows that $a_{\max} \in a.$
  This shows existence of a representative $a_{\max} \in a$ such that $f = a_{\max} \at{{\dom f}}$ for all $f\in a,$
  and it is clear that this condition uniquely determines $a_{\max}.$
\end{proof}

\begin{proposition} \label{proposition:alEvSubring}
  Let $X$ be a topological space and let $\Dom$ be the set of dense open subsets of $X.$
  Moreover, let $\Dom'$ be a $\pi$-subsystem of the topology of $X$ such that $\Dom' \subseteq \Dom.$
  Then the canonical inclusion map $\iota \colon \bigcup_{A \in \Dom'} \Stetig(A) \to \bigcup_{A \in \Dom} \Stetig(A),$ $f\mapsto \iota(f) \coloneqq f$
  descends to a well-defined positive ring morphism $\check\iota \colon \Stetig_\approx(\Dom') \to \Stetig_\approx(\Dom) = \AlEv(X),$
  $[f]' \mapsto \check\iota([f]') \coloneqq [f],$ where $[\argument]' \colon \bigcup_{A \in \Dom'} \Stetig(A) \to \Stetig_\approx(\Dom')$ and
  $[\argument] \colon \bigcup_{A \in \Dom} \Stetig(A) \to \Stetig_\approx(\Dom)$ are the canonical maps onto the quotient.
  Moreover, this positive ring morphism $\check\iota$ is an order embedding.
\end{proposition}
\begin{proof}
  We write $\lesssim'$ and $\lesssim$ for the quasi orders on $\bigcup_{A \in \Dom'} \Stetig(A)$ and $\bigcup_{A \in \Dom} \Stetig(A),$ respectively.
  Consider any $f,g \in \bigcup_{A \in \Dom'} \Stetig(A).$ It is clear that $\iota(f+g) = \iota(f)+\iota(g)$ and $\iota(fg) = \iota(f) \iota(g)$ hold.
  If $f \lesssim' g,$ then certainly also $\iota(f) \lesssim \iota(g).$ Conversely, if $\iota(f) \lesssim \iota(g),$
  then there exists a dense open subset $A$ of $X$ fulfilling $A \subseteq {\dom f} \cap {\dom g}$ and such that $f\at{A} \le g\at{A}$ pointwise.
  As $A$ is dense in $X,$ $A$ is in particular dense in ${\dom f} \cap {\dom g},$ and therefore $f\at{{\dom f} \cap {\dom g}} \le g\at{{\dom f} \cap {\dom g}}$ pointwise.
  As ${\dom f} \cap {\dom g} \in \Dom'$ this shows that $f \lesssim' g.$
  
  It is now straightforward to check that $\iota$ descends to a well-defined map $\check\iota \colon \Stetig_\approx(\Dom') \to \Stetig_\approx(\Dom) = \AlEv(X),$
  $[f]' \mapsto \check\iota([f]') \coloneqq [f],$
  and that $\check \iota$ is a positive ring morphism and an order embedding.
\end{proof}
By slight abuse of notation, dropping the canonical order embedding $\check \iota,$ we can therefore treat any partially ordered 
commutative ring $\Stetig_\approx(\Dom')$ like in the above Proposition~\ref{proposition:alEvSubring} as a subring of $\AlEv(X).$
In this sense, $\Stetig_\approx(\Dom')$ consists of all those elements $a \in \AlEv(X)$ that
have a representative $f \in a$ fulfilling ${\dom f} \in \Dom',$ or equivalently,
\begin{equation}
  \label{eq:alEvSubring}
 \Stetig_\approx(\Dom') = \set[\big]{a \in \AlEv(X)}{\textup{there is $A \in \Dom'$ such that }A \subseteq \dom a_{\max}} 
\end{equation}
where $\dom a_{\max}$ is the maximal domain of an element $a$ of $\AlEv(X)$ like in Proposition~\ref{proposition:maxDomain}.

\section{Extended positive characters} \label{sec:characters}

Let $\R$ be a partially ordered commutative ring, then $\Loc(\R)$ is a multiplicative
submonoid of $\R$ and cancelable, i.e.~whenever $r\in \R,$ $s\in \Loc(\R)$ fulfil $rs = 0,$ then $r=0,$ because $rs = 0 \in \R^+ \cap (-\R^+)$
implies $r \in \R^+ \cap (-\R^+) = \{0\}.$ This simplifies the construction of the \emph{localization $\R_\loc$ of $\R$}
(with respect to $\Loc(\R)$):

$\R_\loc$ is the commutative ring whose underlying set is $\bigl(\R \times \Loc(\R)\bigr)/{\sim},$
the quotient of the cartesian product of the sets $\R$ and $\Loc(\R)$ modulo the equivalence
relation $\sim$ that is defined, for $p,r \in \R$ and $q,s \in \Loc(\R),$ as $(p,q) \sim (r,s)$
if and only if $ps = qr$ (transitivity of $\sim$ holds because $\Loc(\R)$ is cancelable).
For an equivalence class in $\R \times \Loc(\R)$ with respect to $\sim$
we simply write $p/q \in \R_\loc$ with representatives $p\in\R,$ $q\in \Loc(\R).$
Addition and multiplication of $\R_\loc$ are defined as
\begin{equation}
  p/q + r/s \coloneqq (ps+qr)/(qs)\quad\quad\text{and}\quad\quad (p/q)(r/s) \coloneqq (pq)/(rs) \quad\quad\text{for $p/q,r/s \in \R_\loc$.}
\end{equation}
This way, $\R_\loc$ becomes a well-defined commutative ring with unit $1/1 \in \R_\loc.$ Moreover, write
\begin{equation}
  \R^+_\loc \coloneqq \set[\big]{p/q}{p\in \R^+, q \in \Loc(\R)}
  ,
\end{equation}
then $\R^+_\loc$ is a preordering of $\R_\loc$ and an element $p/q \in \R_\loc$ with $p\in \R$ and $q \in \Loc(\R)$
fulfils $p/q \in \R^+_\loc$ if and only if $p\in \R^+$: Indeed, $p\in \R^+$ implies $p/q \in \R^+_\loc$ by definition,
and if $p/q \in \R^+_\loc,$ then there exist $r \in \R^+$ and $s \in \Loc(\R)$ such that $p/q = r/s,$ i.e.~$ps = qr \in \R^+,$
so $p\in \R^+$ by localizability of $s.$ In particular, $p/q = 0 \in \R^+_\loc \cap (-\R^+_\loc)$ implies $p\in \R^+ \cap (-\R^+) = \{0\},$
so $\R^+_\loc \cap (-\R^+_\loc) = \{0\}.$ Therefore $\R_\loc$ becomes a partially ordered commutative ring with
order given, for $p,r \in \R$ and $q,s \in \Loc(\R),$ by $p/q \le r/s$ if and only if $r/s - p/q \in \R^+_\loc,$
or equivalently, $ps \le qr.$

We summarize the preceding discussion:
\begin{proposition} \label{proposition:loc}
  Let $\R$ be a partially ordered commutative ring, then its localization $\R_\loc$ is a partially ordered commutative ring
  and the
  canonical inclusion 
  $\iota \colon \R \to \R_\loc,$ $r\mapsto \iota(r) \coloneqq r/1$ is a positive ring morphism and an order embedding.
\end{proposition}

In the next step we investigate the subring of the localization that is given by its uniformly bounded elements and the corresponding space of positive characters.
So let $\R$ be a partially ordered commutative ring, then we define the subset
\begin{equation}
  \R^\bd_\loc
  \coloneqq
  \set[\big]{
    a\in\R_\loc
  }{
    \textup{there exists $n\in \NN_0$ such that $-n \le a \le n$}
  }
\end{equation}
of \emph{uniformly bounded} elements of $\R_\loc.$ In particular $r / s \in \R^\bd_\loc$ for $r\in \R,$ $s\in \Loc(\R)$ with $-s \le r \le s,$
and $1/s \in (\R^\bd_\loc)^+ = \R^+_\loc \cap \R^\bd_\loc$ for all $s\in \Loc(\R)$.

\begin{proposition}
  Let $\R$ be a partially ordered commutative ring, then $\R^\bd_\loc$ is a subring of $\R_\loc.$
\end{proposition}
\begin{proof}
  This is well-known, just note that, if $a,b\in \R_\loc$ and $m,n\in \NN_0$ fulfil $-m \le a \le m$ and $-n \le b\le n,$ then
  \begin{align*}
    3mn-ab &= (m-a)(n+b) + (m+a)n + m(n-b) \in \R^+_\loc
  \shortintertext{and}
    3mn+ab &= (m+a)(n+b) + (m-a)n + m(n-b) \in \R^+_\loc,
  \end{align*}
  i.e.~$-3mn \le ab \le 3mn.$
\end{proof}
In particular, $\R^\bd_\loc$ is a partially ordered commutative ring with positive cone $(\R^\bd_\loc)^+ = \R^+_\loc \cap \R^\bd_\loc.$
Note that $1/q \in (\R^\bd_\loc)^+$ for all $q\in \Loc(\R)$ because $\Loc(\R) \subseteq 1+\R^+$ by definition.

We can now give another of the central definitions of this article:

\begin{definition}
  Let $\R$ be a partially ordered commutative ring, then an \emph{extended character} of $\R$ is a ring morphism
  from $\R^\bd_\loc$ to $\RR.$ Such an extended character $\varphi \colon \R^\bd_\loc \to \RR$ is called \emph{positive}
  if $\varphi(a) \ge 0$ for all $a\in (\R^\bd_\loc)^+.$ We write
  \begin{equation}
    \KSpace(\R)
    \coloneqq
    \set[\big]{\varphi\colon \R^\bd_\loc \to \RR}{ \varphi \textup{ a positive ring morphism}}
  \end{equation}
  for the topological space of all positive extended characters of $\R^\bd_\loc$ equipped with the weak-$^*$-topology,
  i.e.~the topology obtained from the subbasis of all the (by definition) open subsets
  \begin{equation}
    \Open_{a,V} \coloneqq \set[\big]{\varphi\in \KSpace(\R)}{\varphi(a) \in V} \subseteq \KSpace(\R)
  \end{equation}
  with $a\in \R^\bd_\loc$ and with $V$ an open subset of $\RR.$
\end{definition}

\begin{proposition} \label{proposition:neighbourhoods}
  Let $\R$ be a partially ordered commutative ring with $\NN \subseteq \Loc(\R)$ and let $\varphi \in \KSpace(\R)$ 
  and a neighbourhood $U$ of $\varphi$ in $\KSpace(\R)$ be given.
  Then there exists $b \in \R^\bd_\loc$ such that $\varphi \in \Open_{b,{]{-\infty},0[}} \subseteq U.$
\end{proposition}
\begin{proof}
  Note that $\NN \subseteq \Loc(\R)$ implies $\QQ \subseteq \R^\bd_\loc.$
  There are $n\in \NN,$ elements $a_1, \dots, a_n \in \R^\bd_\loc,$ and open subsets
  $V_1, \dots, V_n$ of $\RR$ such that $\varphi \in \Open_{a_1,V_1} \cap \dots \cap \Open_{a_n,V_n} \subseteq U.$
  In particular $\varphi(a_j) \in V_j$ for all $j\in \{1,\dots,n\}$ and therefore there is $\epsilon \in {]0,\infty[} \cap \QQ$
  such that ${]{\varphi(a_j)-\epsilon},\varphi(a_j)+\epsilon[} \subseteq V_j$ for all $j\in \{1,\dots,n\}.$
  There also are $\lambda_1,\dots,\lambda_n \in \QQ$ such that $\abs{\varphi(a_j)-\lambda_j} \le \min\{ \epsilon/\sqrt{4n} ,\epsilon/4 \}$ for all $j\in \{1,\dots,n\}.$
  Set $b \coloneqq -1 + 2\epsilon^{-2}\sum_{j=1}^n (a_j - \lambda_j)^2 \in \R^\bd_\loc,$ then 
  \begin{equation*}
    \varphi(b) = -1 + 2\epsilon^{-2}\sum_{j=1}^n \abs{\varphi(a_j) - \lambda_j}^2 \le -1 + 2n\epsilon^{-2} \bigr(\epsilon^2/(4n)\bigr) = - 1/2 < 0.
  \end{equation*}
  This shows that $\varphi \in \Open_{b,{]{-\infty},0[}}.$
  
  Moreover, for every $\psi \in \KSpace(\R) \setminus U$
  there is $j' \in \{1,\dots,n\}$ such that $\psi \notin \Open_{a_{j'},V_{j'}},$ i.e.~$\psi(a_{j'}) \notin V_{j'}$
  and in particular $\abs{\psi(a_{j'}) - \varphi(a_{j'})} \ge \epsilon.$ It follows that
  $\abs{\psi(a_{j'}) - \lambda_{j'}} \ge \epsilon - \epsilon/4 = \frac{3}{4}\epsilon,$ so
  \begin{equation*}
    \psi(b) = -1+2\epsilon^{-2}\sum_{j=1}^n \abs{\psi(a_j) - \lambda_j}^2 \ge -1 + 2\epsilon^{-2} \abs{\psi(a_{j'}) - \lambda_{j'}}^2 \ge -1+2 \biggl(\frac{3}{4}\biggr)^2 = 1/8 > 0
    ,
  \end{equation*}
  so $\psi \notin \Open_{b,{]{-\infty},0[}}.$ This shows that $\Open_{b,{]{-\infty},0[}} \subseteq U.$
\end{proof}

\section{The extended Gelfand transformation} \label{sec:gelfand}

To every partially ordered commutative ring $\R$ we can first assign its topological space $\KSpace(\R)$ of
positive extended characters, and, after choosing a suitable $\pi$-subsystem $\LocDomains(\R)$ of the topology of $\KSpace(\R)$,
we will obtain another partially ordered commutative
ring $\Stetig_\approx\bigl(\LocDomains(\R)\bigr).$
We then relate these by constructing a positive ring morphism $\widehat\argument \colon \R \to \Stetig_\approx\bigl(\LocDomains(\R)\bigr),$
the extended Gelfand transformation.

\begin{definition}
  Let $\R$ be a partially ordered commutative ring. For all $s\in \Loc(\R)$ we define the open subset
  $\Open_{s<\infty} \coloneqq \Open_{1/s,{]0,\infty[}}$ of $\KSpace(\R)$ and we write
  $\LocDomains(\R) \coloneqq \set[\big]{\Open_{s<\infty}}{s\in \Loc(\R)}.$
\end{definition}

\begin{proposition} \label{proposition:openQS}
  Let $\R$ be a partially ordered commutative ring, then $\Open_{q<\infty} \cap \Open_{s<\infty} = \Open_{qs<\infty}$
  for all $q,s\in \Loc(\R)$ and $\Open_{1<\infty} = \KSpace(\R),$ so $\LocDomains(\R)$ is a $\pi$-subsystem of the topology of $\KSpace(\R)$.
\end{proposition}
\begin{proof}
  Clearly $\Open_{1<\infty} = \Open_{1/1,{]0,\infty[}} = \KSpace(\R)$.
  Now consider $q,s\in \Loc(\R)$ and $\varphi \in \KSpace(\R).$ Then $\varphi \in \Open_{q<\infty} \cap \Open_{s<\infty}$
  if and only if $\varphi(1/q) > 0$ and $\varphi(1/s)>0.$ Similarly, $\varphi \in \Open_{qs<\infty}$ if and only if $\varphi\bigl(1/(qs)\bigr)>0.$
  If $\varphi(1/q) > 0$ and $\varphi(1/s)>0,$ then of course $\varphi\bigl(1/(qs)\bigr) = \varphi(1/q) \varphi(1/s)>0.$
  Conversely, if $\varphi\bigl(1/(qs)\bigr) = \varphi(1/q) \varphi(1/s)>0,$ then either $\varphi(1/q) > 0$ and $\varphi(1/s)>0,$
  or $\varphi(1/q) < 0$ and $\varphi(1/s)<0.$ The latter, however, is not possible because $\varphi$ is a positive ring morphism
  and $1/q,1/s \in (\R^\bd_\loc)^+.$
\end{proof}

Before defining the extended Gelfand transformation we need to check the following:
\begin{lemma} \label{lemma:equivalentReps}
  Let $\R$ be a partially ordered commutative ring. For $r\in \R$ and $s \in \Loc(\R)$ satisfying $r/s \in \R^\bd_\loc$
  define the function $r_s \colon \Open_{s<\infty} \to \RR,$
  \begin{align}
    \varphi &\mapsto  r_s(\varphi) \coloneqq \varphi(1/s)^{-1} \varphi(r/s)
    .
  \end{align}
  Then $r_s$ is continuous. Moreover, if $r\in \R$ and $s,s' \in \Loc(\R)$ fulfil $r/s, r/s' \in \R^\bd_\loc,$
  then $r_s(\varphi) = r_{s'}(\varphi)$ for all $\varphi \in \Open_{s<\infty} \cap \Open_{s'<\infty}.$
  In particular, $r_s$ and $r_{s'}$ are representatives of the same $\approx$-equivalence class in $\Stetig_\approx\bigl( \LocDomains(\R) \bigr).$
\end{lemma}
\begin{proof}
  Consider $r\in \R$ and $s \in \Loc(\R)$ such that $r/s \in \R^\bd_\loc.$ Then
  $\Open_{s<\infty} \ni \varphi \mapsto \varphi(1/s) \in {]0,\infty[}$
  and $\Open_{s<\infty} \ni \varphi \mapsto \varphi(r/s) \in \RR$ are continuous by definition of the weak $^*$-topology
  on $\KSpace(\R).$ It follows that $r_s$ is also continuous.
  
  Now consider $r\in \R$ and $s,s' \in \Loc(\R)$ satisfying $r/s, r/s' \in \R^\bd_\loc.$
  Then for all $\varphi \in \Open_{s<\infty} \cap \Open_{s'<\infty}$ the identity
  \begin{align*}
    \varphi(1/s)^{-1} \varphi(r/s)
    &=
    \varphi(1/s)^{-1} \varphi(1/s')^{-1}\varphi(1/s')\, \varphi(r/s)
    \\
    &=
    \varphi(1/s)^{-1} \varphi(1/s')^{-1}\varphi(1/s)\, \varphi(r/s')
    \\
    &=
    \varphi(1/s')^{-1} \varphi(r/s')
  \end{align*}
  holds. As $\Open_{s<\infty}, \Open_{s'<\infty} \in \LocDomains(\R)$ by definition, therefore
  $\Open_{s<\infty} \cap \Open_{s'<\infty} \in \LocDomains(\R)$ by the previous Proposition~\ref{proposition:openQS}, it follows that
  $r_s$ and $r_{s'}$ are representatives of the same $\approx$-equivalence class in $\Stetig_\approx\bigl( \LocDomains(\R) \bigr).$
\end{proof}
Note also that for any element $r$ of a localizable partially ordered commutative ring $\R$ there is $s\in \Loc(\R)$
such that $-s \le r \le s,$ hence $r/s \in \R^\bd_\loc.$ This observation and Lemma~\ref{lemma:equivalentReps} allow us to define:

\begin{definition} \label{definition:widehat}
  Let $\R$ be a localizable partially ordered commutative ring, then the \emph{extended Gelfand transformation}
  is the map $\widehat \argument \colon \R \to \Stetig_\approx\bigl( \LocDomains(\R) \bigr),$
  $r \mapsto \widehat r$ that assigns to every element $r\in \R$ the $\approx$\=/equivalence class of the function
  $r_s \in \Stetig( \Open_{s<\infty} )$ from the previous Lemma~\ref{lemma:equivalentReps},
  where $s\in \Loc(\R)$ is any element satisfying $r/s\in\R^\bd_\loc.$
\end{definition}

\begin{proposition} \label{proposition:widehat}
  Let $\R$ be a partially ordered commutative ring, then the extended Gelfand transformation
  $\widehat \argument \colon \R \to \Stetig_\approx\bigl(\LocDomains(\R)\bigr)$ is a positive ring morphism.
\end{proposition}
\begin{proof}
  Clearly $\widehat 1$ is the equivalence class of the constant $1$-function on $\KSpace(\R)$.
  Given $r\in \R^+$ and $s\in \Loc(\R)$ such that $r/s \in \R^\bd_\loc,$
  then $\varphi(1/s)^{-1} \varphi(r/s) \ge 0$ for all $\varphi \in \Open_{s < \infty}$
  because $1/s,r/s \in (\R^\bd_\loc)^+,$ so $\widehat r \in \Stetig_\approx(\LocDomains(\R))^+.$
  Now consider $p,r\in \R$ and $q,s \in \Loc(\R)$ such that $p/q,r/s\in \R^\bd_\loc$.
  Then for all $\varphi\in \Open_{q<\infty} \cap \Open_{s<\infty}$ the identities
  \begin{align*}
    \varphi(1/q)^{-1}\varphi(p/q) + \varphi(1/s)^{-1}\varphi(r/s)
    &=
    \varphi(1/q)^{-1}\varphi(1/s)^{-1}\bigl(\varphi(1/s)\varphi(p/q) + \varphi(1/q)\varphi(r/s)\bigr)
    \\
    &=
    \varphi\bigl(1/(qs)\bigr)^{-1} \varphi\bigl((p+r)/(qs)\bigr)
  \shortintertext{and}
    \varphi(1/q)^{-1}\varphi(p/q) \, \varphi(1/s)^{-1}\varphi(r/s)
    &=
    \varphi\bigl(1/(qs)\bigr)^{-1} \varphi\bigl((pr)/(qs)\bigr)
  \end{align*}
  hold, i.e.~$p_q(\varphi) + r_s(\varphi) = (p+r)_{qs}(\varphi)$ and $p_q(\varphi) r_s(\varphi) = (pr)_{qs}(\varphi).$
  So $\widehat{p}+\widehat{r} = \widehat{p{+}r}$ and $\widehat{p}\,\widehat{r} = \widehat{pr}.$
\end{proof}

While the extended Gelfand transformation is not an order embedding in general, we can precisely describe the defect:

\begin{theorem} \label{theorem:positivity}
  Let $\R$ be a localizable partially ordered commutative ring with $\NN \subseteq \Loc(\R).$
  Then $\KSpace(\R)$ is a compact Hausdorff space and
  \begin{equation}
    \set[\big]{r\in \R}{\widehat r \in \Stetig_\approx( \LocDomains(\R) )^+} = (\R^+)^\ddagger
  \end{equation}
  with the operation $\argument^\ddagger$ from \eqref{eq:ddagger}.
\end{theorem}
\begin{proof}
  This is an application of the Positivstellensatz for archimedean preorderings \cite[Sec.~``Préordres archimédiens'']{krivine:positivstellensatz}
  (see also e.g.~\cite[Sec.~5]{marshall:positivePolynomialsAndSumsOfSquares} for a more modern reference) to the preordering $(\R^\bd_\loc)^+$ of $\R^\bd_\loc$:
  From $\NN \subseteq \Loc(\R)$ it follows that $\QQ \subseteq \R^\bd_\loc.$ By definition of $\R^\bd_\loc,$
  for every $a \in \R^\bd_\loc$ there exists $n\in \NN_0$ such that $-n \le a \le n$ with respect to the order of
  $\R_\loc$ or, equivalently, the order of $\R^\bd_\loc.$ In other words, the preordering $(\R^\bd_\loc)^+$ of $\R^\bd_\loc$
  is ``archimedean'' in the sense of \cite[Def.~5.2.1]{marshall:positivePolynomialsAndSumsOfSquares}. Therefore, by \cite[Thm.~5.7.2]{marshall:positivePolynomialsAndSumsOfSquares}, 
  $\KSpace(\R)$ is a compact Hausdorff space and, whenever an element $a \in \R^\bd_\loc$ fulfils $\varphi(a) > 0$ for all $\varphi \in \KSpace(\R),$
  then $a \in (\R^\bd_\loc)^+ = \R^+_\loc \cap \R^\bd_\loc.$
  
  So assume that $\widehat r \in \Stetig_\approx( \LocDomains(\R) )^+,$ i.e.~for any $s\in \Loc(\R)$
  fulfilling $r/s \in \R^\bd_\loc$ the representative $r_s$ of $\widehat r$ like in Lemma~\ref{lemma:equivalentReps} fulfils $0\lesssim r_s.$
  This means that there is $q \in \Loc(\R)$
  such that $\Open_{q<\infty} \subseteq \Open_{s<\infty}$ and
  $0 \le r_s(\varphi) = \varphi(1/s)^{-1} \varphi(r/s)$ for all $\varphi \in \Open_{q<\infty},$
  so $0 \le \varphi(r/s)$ for $\varphi \in \Open_{q<\infty}.$
  Then $r/(qs) \in \R^\bd_\loc$ and
  $\varphi(r/(qs)) = \varphi(1/q) \varphi(r/s) \ge 0$ for all $\varphi \in \KSpace(\R)$:
  Indeed, by positivity of $\varphi$ either $\varphi(1/q) = 0$ or $\varphi(1/q) > 0$ holds, and in the latter case $\varphi \in \Open_{q<\infty}$
  so that $\varphi(r/s)\ge 0.$
  Therefore $\varphi\bigl(1/n+r/(qs)\bigr) \ge 1/n > 0$ for all $n\in \NN$ and all $\varphi \in \KSpace(\R),$
  so the Positivstellensatz for archimedean preorderings shows that $(nr+qs)/(nqs) = 1/n+r/(qs) \in (\R^\bd_\loc)^+ \subseteq \R^+_\loc$
  for all $n\in \NN,$ i.e.~$nr+qs \in \R^+$ for all $n\in \NN$ by Proposition~\ref{proposition:loc}, and therefore $r \in (\R^+)^\ddagger.$

  Conversely, assume that $r \in (\R^+)^\ddagger,$ then there is $s \in \R$ such that $nr+s \in \R^+$ holds for all $n\in \NN.$
  By localizability of $\R$ there exists $t \in \Loc(\R)$ such that $1+r^2 + s^2 \le t.$ Then
  $-t \le r \le t$ because $2t\pm 2r = t + (t-(1+r^2+s^2)) + (1\pm r)^2 + s^2 \in \R^+$ for both choices of the sign $\pm$
  and because $2\in \Loc(\R)$ by assumption, and similarly $-t \le s \le t.$ This shows that $r/t, s/t \in \R^\bd_\loc$
  and therefore also $(nr+ s)/t = n(r/t) + s/t \in \R^\bd_\loc$ for all $n \in \NN_0.$ Consequently, every $\varphi\in \KSpace(\R)$
  fulfils $\varphi(r/t) = \frac{1}{n} \varphi\bigl((nr+s)/t\bigr) - \frac{1}{n} \varphi(s/t) \ge -\frac{1}{n} \varphi(s/t)$
  for all $n\in \NN,$ and therefore $\varphi(r/t) \ge 0.$ It follows that the representative
  $r_t \colon \Open_{t<\infty} \to \RR$ of $\widehat r$
  is pointwise positive on $\Open_{t<\infty},$ which shows that $\widehat r \in \Stetig_\approx( \LocDomains(\R) )^+.$
\end{proof}
This result is similar to some of the strict Positivstellensätze for preorderings like \cite{marshall:extendingArchimedeanPositivstellensatzToTheNonCOmpactCase},
but without any assumption of boundedness or finiteness.

\begin{corollary}
  Let $\R$ be a localizable partially ordered commutative ring with $\NN \subseteq \Loc(\R).$
  Then the kernel of the extended Gelfand transformation $\widehat{\argument} \colon \R \to \Stetig_\approx( \LocDomains(\R) )^+$ is
  \begin{equation}
    \ker \widehat\argument = \set[\big]{r\in \R}{\textup{there is $s\in \R^+$ such that }{-s} \le nr \le s\textup{ for all }n\in \NN}
  \end{equation}
\end{corollary}
\begin{proof}
  By the previous Theorem~\ref{theorem:positivity}, $\ker \widehat\argument = (\R^+)^\ddagger \cap (-(\R^+)^\ddagger).$
  Consider any $r \in (\R^+)^\ddagger \cap (-(\R^+)^\ddagger),$ then by definition there are $e,f\in \R$ such that $e+nr \in \R^+$ and $f-nr \in \R^+$ for all $n\in \NN.$
  So set $s \coloneqq e+f = (e+r)+(f-r) \in \R^+,$ then $-s = -e -f \le -e -r \le nr \le  f-r \le e+f = s$ for all $n\in \NN.$
  This proves the inclusion $\subseteq,$ and the converse inclusion follows immediately from the identity
  $\ker \widehat\argument = (\R^+)^\ddagger \cap (-(\R^+)^\ddagger)$ and from the definition of the operation $\argument^\ddagger.$
\end{proof}

\section{Representation by almost everywhere defined continuous functions} \label{sec:main}
We can almost put all the pieces together and prove the main theorem. The only missing part is the following observation:

\begin{proposition} \label{proposition:dense}
  Let $\R$ be a localizable archimedean partially ordered commutative ring and consider any element $q\in \Loc(\R).$
  Then $\Open_{q<\infty}$ is dense in $\KSpace(\R).$
\end{proposition}
\begin{proof}
  Note that $\NN \subseteq \Loc(\R)$ by Lemma~\ref{lemma:archLoc} so that Proposition~\ref{proposition:neighbourhoods} applies.
  Assume to the contrary that there are $\varphi \in \KSpace(\R)$ and a neighbourhood $U$ of $\varphi$ in $\KSpace(\R)$
  such that $U \cap \Open_{q<\infty} = \emptyset.$ By Proposition~\ref{proposition:neighbourhoods}
  there is $b = r/s \in \R^\bd_\loc$ with $r\in \R,$ $s\in \Loc(\R)$ such that $\varphi \in \Open_{b,{]{-\infty},0[}} \subseteq U.$
  In particular, $\Open_{b,{]{-\infty},0[}} \cap \Open_{q<\infty} = \emptyset,$
  so $\psi\bigl(r/(sq)\bigr) = \psi(b) \psi(1/q) \ge 0$ for all $\psi \in \Open_{q<\infty}$.
  As $\Open_{sq<\infty} = \Open_{s<\infty} \cap \Open_{q<\infty} \subseteq \Open_{q<\infty}$ by Proposition~\ref{proposition:openQS},
  the representative $r_{sq} \in \widehat r$, $\Open_{sq<\infty} \ni \psi \mapsto \psi(1/(sq))^{-1} \psi\bigl(r/(sq)\bigr) \in \RR$
  fulfils $r_{sq} \ge 0$ pointwise, so $\widehat r \in \Stetig_\approx\bigl( \LocDomains(\R)\bigr)^+.$
  Theorem~\ref{theorem:positivity} now shows that $r \in (\R^+)^\ddagger.$ As $\R$ is archimedean by assumption,
  this means that $r \in \R^+,$ i.e.~$b \in (\R^\bd_\loc)^+.$ But this would
  imply $\psi(b) \ge 0$ for all $\psi \in \KSpace(\R),$ in particular $\varphi(b) \ge 0,$ which is impossible 
  because $\varphi \in \Open_{b,{]{-\infty},0[}}.$
\end{proof}
As an immediate consequence we obtain:

\begin{corollary} \label{corollary:dense}
  Let $\R$ be a localizable archimedean partially ordered commutative ring,
  then, in the sense of Proposition~\ref{proposition:alEvSubring} and Equation \eqref{eq:alEvSubring},
  $\Stetig_\approx\bigl(\LocDomains(\R)\bigr) \subseteq \AlEv\bigl(\KSpace(\R)\bigr).$
  So by slight abuse of notation we can treat the extended Gelfand transformation 
  as a positive ring morphism $\widehat{\argument} \colon \R \to \AlEv\bigl(\KSpace(\R)\bigr).$
\end{corollary}

We finally state and prove the main theorem:

\begin{theorem} \label{theorem:main}
  Let $\R$ be a partially ordered commutative ring, then the following are equivalent:
  \begin{enumerate}
    \item \label{item:main:iso}
      There exists a compact Hausdorff space $X$ such that $\R$ is isomorphic (as a partially ordered commutative ring)
      to a subring of the partially ordered commutative ring $\AlEv(X)$ of almost everywhere defined continuous
      $\RR$\=/valued functions on $X.$
    \item \label{item:main:orderembed}
      $\R$ is localizable and the extended Gelfand transformation $\widehat{\argument} \colon \R \to \Stetig_\approx\bigl(\LocDomains(\R)\bigr)$
      is an order embedding.
    \item \label{item:main:archloc} $\R$ is archimedean and localizable.
    \item \label{item:main:archlocStrong} $\R$ is archimedean and strongly localizable.
  \end{enumerate}
  If these equivalent statements hold, then $\Stetig_\approx\bigl(\LocDomains(\R)\bigr) \!\subseteq\! \AlEv\bigl(\KSpace(\R)\bigr)$
  like in the previous Corollary~\ref{corollary:dense}.
\end{theorem}
\begin{proof}
  Assume that $\R$ is archimedean and localizable. Consider $r\in \R$ such that $\widehat{r} \in \Stetig_\approx(\LocDomains(\R))^+\!,$
  then $r\in (\R^+)^\ddagger$ by Theorem~\ref{theorem:positivity}, so $r\in \R^+$ due to the assumption that $\R$
  is archimedean. This shows that the extended Gelfand transformation $\widehat{\argument} \colon \R \to \Stetig_\approx\bigl(\LocDomains(\R)\bigr)$
  is an order embedding. As the extended Gelfand transformation is always a ring morphism by Proposition~\ref{proposition:widehat},
  this means that $\R$ and its image under $\widehat{\argument}$ are isomorphic as partially ordered commutative rings.
  By the previous Corollary~\ref{corollary:dense}, $\Stetig_\approx\bigl(\LocDomains(\R)\bigr) \subseteq \AlEv\bigl(\KSpace(\R)\bigr),$
  so $\R$ is isomorphic as a partially ordered commutative ring to a subring of $\AlEv\bigl(\KSpace(\R)\bigr),$ with 
  $\KSpace(\R)$ a compact Hausdorff space by Theorem~\ref{theorem:positivity}. This shows that \refitem{item:main:archloc}
  implies both \refitem{item:main:orderembed} and \refitem{item:main:iso}.
  
  Conversely, we first show that \refitem{item:main:orderembed} implies \refitem{item:main:archloc}, so assume that
  \refitem{item:main:orderembed} holds. Then $\R$ is archimedean: Given $r,s\in \R$ such that $nr+s \in \R^+$
  for all $n\in \NN,$ i.e.~$r\in (\R^+)^\ddagger,$ then $\widehat{r} \in \Stetig_\approx( \LocDomains(\R))^+$
  by Theorem~\ref{theorem:positivity}. As $\widehat{\argument} \colon \R \to \Stetig_\approx\bigl(\LocDomains(\R)\bigr)$
  is an order embedding by assumption, it follows that $r\in \R^+.$
  
  Finally, \refitem{item:main:iso} implies \refitem{item:main:archlocStrong}, because for any compact Hausdorff space $X,$
  the partially ordered commutative ring $\AlEv(X)$ and all its subrings are archimedean and strongly localizable by Propositions~\ref{proposition:alEvArchLoc}.
  Proposition~\ref{proposition:weakstronloc} shows that \refitem{item:main:archlocStrong} implies \refitem{item:main:archloc}.
\end{proof}

\begin{remark}
  For the implication \refitem{item:main:archloc}~$\implies$~\refitem{item:main:iso} in the previous Theorem~\ref{theorem:main},
  the assumption that the multiplication of the archimedean localizable partially ordered commutative ring $\R$ is associative and commutative is
  redundant by \cite{schoetz:arxivAssociativityAndCommutativityOfPartiallyOrderedRings}, and the assumption that $\R^+$ contains all squares
  is redundant by Proposition~\ref{proposition:semiring}:
  Consider an archimedean partially ordered abelian group $\R$ endowed with a biadditive binary operation $\mu \colon \R \times \R \to \R$
  that admits a (necessarily unique) neutral element $1\in \R^+$ and that satisfies $\mu(u,v) \in \R^+$ for all $u,v \in \R^+.$
  We adapt localizability to this more general setting by assuming that for all $r\in \R$ there exists $s\in 1+\R^+$
  such that $-s \le r \le s$ and such that $s$ is localizable in the following sense: Whenever some element $r'\in \R$
  fulfils $\mu(r',s) \in \R^+$ or $\mu(s,r') \in \R^+,$ then $r' \in \R^+.$ Then the restriction of $\mu$ to $\R^+$ is associative and commutative by 
  \cite[Cor.~13, Props.~21-22]{schoetz:arxivAssociativityAndCommutativityOfPartiallyOrderedRings},
  and therefore $\mu$ is associative and commutative on whole $\R.$ So Proposition~\ref{proposition:semiring} shows that
  $r^2 \in \R^+$ for all $r\in \R$, i.e.~$\R$ in particular is a partially ordered commutative ring, and Theorem~\ref{theorem:main} applies.
\end{remark}

By Proposition~\ref{proposition:maxDomain}, every equivalence class of an almost everywhere defined continuous function $a$ on a topological space $X$
has a unique representative $a_{\max} \in a$ with maximal domain. This allows to decide whether or not that function can be defined at some given point of $X.$
We apply this to the extended Gelfand transformation:

\begin{definition}
  Let $\R$ be a localizable archimedean partially ordered commutative ring.
  Then for any positive extended character $\varphi \in \KSpace(\R)$ we define a map ${\widehat \varphi} \colon {\dom \widehat \varphi} \to \RR$
  with domain
  \begin{align}
    {\dom \widehat \varphi} &\coloneqq \set{r\in \R}{\varphi \in {\dom \widehat r_{\max}}}
  \shortintertext{that is given by}
    \widehat \varphi(r) &\coloneqq \widehat r_{\max}(\varphi)\quad\quad\text{for all $r\in {\dom \widehat \varphi}.$}
  \end{align}
  Here we identify $\Stetig_\approx\bigl(\LocDomains(\R)\bigr)$ with a subring
  of $\AlEv\bigl(\KSpace(\R)\bigr)$ as discussed in Proposition~\ref{proposition:alEvSubring} and Corollary~\ref{corollary:dense}
  and treat the extended Gelfand transformation as a map 
  $\widehat{\argument} \colon \R \to \Stetig_\approx\bigl(\LocDomains(\R)\bigr) \subseteq \AlEv\bigl(\KSpace(\R)\bigr).$
\end{definition}

\begin{lemma} \label{lemma:extendedCharakterEvaluation}
  Let $\R$ be a localizable archimedean partially ordered commutative ring and $s\in \Loc(\R),$
  then $\Open_{s<\infty} = \dom \widehat s_{\max}.$
\end{lemma}
\begin{proof}
  Clearly $s/s = 1/1 \in \R^\bd_\loc,$ so $s_s \in \Stetig(\Open_{s<\infty})$ is a representative
  of $\widehat s,$ and $\Open_{s<\infty} \subseteq \dom \widehat s_{\max}.$
  Conversely, for $n\in \NN$ set $U_n \coloneqq \set{\varphi \in \KSpace(\R)}{\varphi(1/s) < 1/n }.$
  Then $U_n \cap {\dom \widehat s_{\max}} \cap \Open_{s<\infty}$ is dense in $U_n \cap {\dom \widehat s_{\max}}$
  because $\Open_{s<\infty}$ is dense in $\KSpace(\R)$ by Proposition~\ref{proposition:dense} and
  because $U_n \cap {\dom \widehat s_{\max}}$ is open in $\KSpace(\R).$
  Moreover, for all $\varphi \in U_n \cap {\dom \widehat s_{\max}} \cap \Open_{s<\infty}$ the estimate
  $\widehat s_{\max}(\varphi) = s_s(\varphi) = \varphi(1/s)^{-1} > n$ holds, and consequently
  $\widehat{s}_{\max}(\varphi) \ge n$ for all $\varphi \in U_n \cap {\dom \widehat s_{\max}}$
  by continuity of $\widehat{s}_{\max}$.
  Clearly $\bigcap\nolimits_{n\in \NN} U_n = \KSpace(\R) \setminus \Open_{s<\infty}$ and so it follows that
  \begin{equation*}
    \bigl( \KSpace(\R) \setminus \Open_{s<\infty} \bigr) \cap {\dom \widehat s_{\max}}
    =
    \Bigl( \bigcap\nolimits_{n\in \NN} U_n \Bigr) \cap {\dom \widehat s_{\max}}
    =
    \bigcap\nolimits_{n\in \NN} \bigl( U_n \cap {\dom \widehat s_{\max}} \bigr)
    =
    \emptyset
    .
  \end{equation*}
  This shows that ${\dom \widehat s_{\max}} \subseteq \Open_{s<\infty}.$
\end{proof}

\begin{proposition} \label{proposition:extendedCharakterEvaluation}
  Let $\R$ be a localizable archimedean partially ordered commutative ring and consider $\varphi \in \KSpace(\R).$
  Then $\dom \widehat\varphi$ is a subring of $\R$ and ${\widehat\varphi} \colon {\dom \widehat\varphi} \to \RR$ is a positive ring morphism.
  Moreover, if $r\in \R$ and $s\in {\dom \widehat\varphi} \cap \Loc(\R)$ fulfil $r/s \in \R^\bd_\loc,$ then also $r\in {\dom \widehat\varphi}.$
\end{proposition}
\begin{proof}
  Clearly $1 \in {\dom \widehat \varphi}$ because $\widehat 1_{\max}$ is the constant $1$-function on $\KSpace(\R),$
  and $\widehat\varphi(1) = \widehat 1_{\max}(\varphi) = 1$.
  Consider $r,s \in {\dom \widehat \varphi}$ and write $A \coloneqq \dom \widehat r_{\max},$ $B \coloneqq \dom \widehat s_{\max},$
  then $\varphi \in A \cap B.$ Note that
  $\widehat{r}_{\max}\at{A \cap B} + \widehat{s}_{\max}\at{A\cap B} \in \Stetig(A\cap B)$ is a representative of
  $\widehat{r+s} = \widehat{r}+\widehat{s},$ and $\widehat{r}_{\max}\at{A \cap B} \widehat{s}_{\max}\at{A\cap B} \in \Stetig(A\cap B)$
  is a representative or $\widehat{rs} = \widehat{r}\widehat{s}.$ This shows that $A \cap B \subseteq \dom (\widehat{r+s})_{\max}$
  and $A \cap B \subseteq \dom (\widehat{rs})_{\max},$ so $r+s \in \dom \widehat\varphi$ and $rs \in \dom \widehat\varphi.$
  It also follows that 
  \begin{align*}
    \widehat{\varphi}(r+s)
    &=
    (\widehat{r+s})_{\max}(\varphi)
    =
    \bigl(\widehat{r}_{\max}\at{A \cap B} + \widehat{s}_{\max}\at{A\cap B}\bigr)(\varphi)
    =
    \widehat{r}_{\max}(\varphi) + \widehat{s}_{\max}(\varphi)
    =
    \widehat{\varphi}(r) + \widehat{\varphi}(s)
  \shortintertext{and}
    \widehat{\varphi}(rs)
    &=
    (\widehat{rs})_{\max}(\varphi)
    =
    \bigl(\widehat{r}_{\max}\at{A \cap B}  \widehat{s}_{\max}\at{A\cap B}\bigr)(\varphi)
    =
    \widehat{r}_{\max}(\varphi)  \widehat{s}_{\max}(\varphi)
    =
    \widehat{\varphi}(r) \widehat{\varphi}(s)
    .
  \end{align*}
  So $\dom \widehat\varphi$ is a subring of $\R$ and ${\widehat\varphi} \colon {\dom \widehat\varphi} \to \RR$ is a ring morphism.
  
  Next consider $r\in ({\dom\widehat\varphi})^+ = ({\dom\widehat\varphi}) \cap \R^+,$ then $\widehat r \in \AlEv(\KSpace(\R))^+$
  because the extended Gelfand transformation is positive by Proposition~\ref{proposition:widehat},
  so there exists a representative $f\in \widehat r$ that is pointwise positive on its domain $\dom f.$
  As $\dom f$ is a dense open subset of $\KSpace(\R)$ it is also dense in $\dom \widehat r_{\max},$ so 
  $\widehat r_{\max}(\psi) \ge 0$ for all $\psi \in \dom \widehat r_{\max}$ by continuity of $\widehat r_{\max}$
  and in particular $\widehat\varphi(r) = \widehat r_{\max}(\varphi) \ge 0.$
  
  Finally, assume two elements $r\in \R$ and $s\in {\dom \widehat\varphi} \cap \Loc(\R)$ fulfil $r/s \in \R^\bd_\loc.$
  Then $r_s \in \Stetig(\Open_{s<\infty})$ is a representative of $\widehat r$ and in particular $\Open_{s<\infty} \subseteq {\dom \widehat r_{\max}}.$
  The previous Lemma~\ref{lemma:extendedCharakterEvaluation} shows that
  $\Open_{s<\infty} = {\dom \widehat s_{\max}},$ so $\varphi \in {\dom \widehat s_{\max}} \subseteq {\dom \widehat r_{\max}},$
  i.e.~$r \in \dom \widehat\varphi.$
\end{proof}

\section{Applications} \label{sec:app}
In the $\sigma$-bounded or lattice-ordered case, the main Theorem~\ref{theorem:main} provides new proofs of known results.
The applications to partially ordered fields and to commutative operator algebras seem to be new.

\subsection{The \texorpdfstring{$\sigma$}{sigma}-bounded case}
We adapt \cite[Def.~3.1]{schoetz:gelfandNaimarkTheorems} to partially ordered commutative rings:
\begin{definition}
  A partially ordered commutative ring is \emph{$\sigma$-bounded} if there exists a sequence $(s_n)_{n\in \NN}$ in $\R^+$
  such that for all $r\in \R$ there is $n\in \NN$ for which $-s_n \le r \le s_n$ holds.
\end{definition}
Recall the construction of the maximal domain $\dom f_{\max} \subseteq X$ of an element $f \in \AlEv(X)$ from Proposition~\ref{proposition:maxDomain}
for $X$ any topological space.

\begin{theorem} \label{theorem:sigma}
  Let $\R$ be a $\sigma$-bounded localizable archimedean partially ordered commutative ring, then  
  $\KSpaceFin(\R) \coloneqq \bigcap_{r\in \R} {\dom \widehat r_{\max}}$
  is a dense subset of $\KSpace(\R)$ and the map $\widehatFin{\argument} \colon \R \to \Stetig\bigl( \KSpaceFin(\R) \bigr),$
  \begin{equation*}
    r \mapsto \widehatFin{r} \coloneqq \widehat r_{\max}\at[\big]{\KSpaceFin(\R)}
  \end{equation*}
  is a positive ring morphism and an order embedding.
\end{theorem}
\begin{proof}
  By Proposition~\ref{proposition:extendedCharakterEvaluation} all the maps $\R \ni r \mapsto \widehatFin{r}(\varphi) = \widehat \varphi(r) \in \RR$
  with $\varphi \in \KSpaceFin(\R)$ are positive ring morphisms, so $\widehatFin{\argument} \colon \R \to \Stetig\bigl( \KSpaceFin(\R) \bigr)$
  is a positive ring morphism.
  
  By $\sigma$-boundedness of $\R$ there is a sequence $(s'_n)_{n\in \NN}$ in $\R^+$
  with the property that for all $r\in \R$ there exists $n\in \NN$ for which $-s'_n \le r \le s'_n$ holds.
  By localizability of $\R,$ for every $n\in \NN$ there is $s_n \in \Loc(\R)$ satisfying $s_n' \le s_n.$
  It follows that for all $r\in \R$ there is $n\in \NN$ such that $r / s_n \in \R^\bd_\loc,$ hence $\dom \widehat r_{\max} \supseteq\dom \bigl((\widehat {s_n})_{\max}\bigr)$
  by Proposition~\ref{proposition:extendedCharakterEvaluation}, and therefore
  $\KSpaceFin(\R) = \bigcap_{r\in \R} {\dom \widehat r_{\max}} = \bigcap_{n\in \NN} {\dom \bigl((\widehat {s_n})_{\max}\bigr)}.$
  As $\KSpace(\R)$ is a compact Hausdorff space by Theorem~\ref{theorem:positivity} and in particular a Baire space,
  and as $\dom \bigl( (\widehat {s_n})_{\max}\bigr)$ is a dense open subset of $\KSpace(\R)$ for all $n\in \NN$, it follows that $\KSpaceFin(\R)$ is
  dense in $\KSpace(\R).$
  
  Now consider $r\in \R \setminus \R^+,$ then $\widehat r \in \AlEv\bigl(\KSpace(\R)\bigr) \setminus \AlEv\bigl(\KSpace(\R)\bigr)^+$
  because the extended Gelfand transformation $\widehat{\argument}$ is an order embedding by Theorem~\ref{theorem:main}.
  So $\set{\varphi \in {\dom \widehat r_{\max}}}{\widehat r_{\max}(\varphi) < 0}$ is a non-empty open subset of $\KSpace(\R)$
  and has non-empty intersection with $\KSpaceFin(\R)$ by denseness of $\KSpaceFin(\R)$ in $\KSpace(\R).$
  Therefore $\widehat{r}_{\max} \at{\KSpaceFin(\R)} \in \Stetig\bigl(\KSpaceFin(\R)\bigr) \setminus \Stetig\bigl(\KSpaceFin(\R)\bigr)^+.$
  This shows that the positive ring morphism $\widehatFin{\argument} \colon \R \to \Stetig\bigl( \KSpaceFin(\R) \bigr)$ is an order embedding.
\end{proof}
This representation theorem for $\sigma$-bounded localizable archimedean partially ordered commutative rings is essentially
\cite[Thm.~4.24]{schoetz:gelfandNaimarkTheorems}.

\subsection{The lattice-ordered case}
A \emph{commutative $f$-ring} is a partially ordered commutative ring $\R$ such that for all $r,s\in \R$ the supremum
$r \vee s \coloneqq \sup\{r,s\} \in \R$ and infimum $r \wedge s \coloneqq \inf\{r,s\} \in \R$ exist (i.e.~$\R$ is lattice-ordered)
and such that the following condition is fulfilled: Whenever $r,s\in \R$ fulfil $r\wedge s = 0,$ then $(rt) \wedge s = 0$
for all $t\in \R^+.$ Note that, in contrast to the original definition \cite[Sec.~8]{birkhoff.pierce:latticeOrderedRings},
we here require rings to have a unit. Moreover, by Eq.~(18) in \cite[Sec.~8]{birkhoff.pierce:latticeOrderedRings},
the axiom of positivity of squares is actually redundant for $f$\=/rings.

\begin{theorem}
  Let $\R$ be an archimedean commutative $f$-ring (with multiplicative unit), then $\R$ is localizable and the extended
  Gelfand transformation $\widehat{\argument} \colon \R \to \AlEv\bigl( \KSpace(\R) \bigr)$
  is a positive ring morphism and an order embedding.
\end{theorem}
\begin{proof}
  In \cite[Prop.~26]{schoetz:arxivAssociativityAndCommutativityOfPartiallyOrderedRings} it was shown that all $f$-rings 
  (in particular commutative $f$-rings with multiplicative unit) are strongly localizable, so Theorem~\ref{theorem:main} applies.
\end{proof}
This representation theorem for (unital) archimedean commutative $f$-rings is essentially \cite[Thm.~2.3]{henriksen.johnson:structureOfArchimedeanLatticeOrderedAlgebras}.

\subsection{Partially ordered fields}
We define a \emph{partially ordered field} as a partially ordered commutative ring $\FF$ in which every element $f\in \FF \setminus \{0\}$
has a multiplicative inverse $f^{-1}.$ Note that $f^{-1} = (f^{-1})^2 f \in \FF^+$ for all $f\in \FF^+ \setminus \{0\}.$
As $-1 \notin \FF^+$ because $1 \in \FF^+,$ this in particular means that $(1+f)^{-1} \in \FF^+$ exists for all $f\in \FF^+,$
and therefore every partially ordered field is strongly localizable. Applying the main Theorem~\ref{theorem:main} yields:

\begin{theorem}
  Let $\FF$ be an archimedean partially ordered field, then $\FF$ is localizable and the extended
  Gelfand transformation $\widehat{\argument} \colon \FF \to \AlEv\bigl( \KSpace(\FF) \bigr)$ is a positive
  ring morphism and an order embedding.
\end{theorem}
A typical example of a partially ordered field that is a subring of the partially ordered commutative ring $\AlEv(X)$
for some compact Hausdorff space $X$ is the field of rational functions on $\RR$ (or on the $1$-point compactification of $\RR$).

\subsection{Commutative operator algebras}
Consider a dense linear subspace $\Dom$ of a complex Hilbert space $\Hilb,$ with inner product denoted by $\skal{\argument}{\argument} \colon \Hilb\times\Hilb \to\CC$
(antilinear in the first, linear in the second argument). An \emph{adjointable endomorphism} of $\Dom$ is a (necessarily linear)
map $a\colon \Dom \to \Dom$ for which there exists another (necessarily linear) map $a^* \colon \Dom \to \Dom$ such that
$\skal{\phi}{a(\psi)} = \skal{a^*(\phi)}{\psi}$ holds for all $\phi,\psi \in \Dom.$ This map $a^*,$ if it exists, is uniquely determined, adjointable,
and $(a^*)^* = a.$ We write $\Adbar(\Dom)$ for the set of all adjointable endomorphisms of $\Dom,$ then $\Adbar(\Dom)$
with pointwise addition, pointwise scalar multiplication, and composition as multiplication, is a complex algebra with unit
$\Unit \coloneqq \id_\Dom \colon \Dom \to \Dom,$ $\phi \mapsto \id_\Dom(\phi) \coloneqq \phi.$
The map $\argument^* \colon \Adbar(\Dom) \to \Adbar(\Dom),$ $a\mapsto a^*$ is an antilinear involution fulfilling $(ab)^* = b^* a^*$
for all $a,b\in \Adbar(\Dom).$ In other words, $\Adbar(\Dom)$ is a $^*$\=/algebra. A \emph{$^*$\=/subalgebra} of $\Adbar(\Dom)$
is a subalgebra $\A$ of $\Adbar(\Dom)$ with $\Unit \in \A$ that is stable under the $^*$\=/involution, i.e.\ $a^* \in \A$
for all $a\in \A.$

Now let $\A$ be a commutative $^*$\=/subalgebra of $\Adbar(\Dom),$ then its space of \emph{hermitian elements} $\A_\hermitian \coloneqq \set{a\in \A}{a^* = a}$
is a commutative ring (even a commutative real algebra) and carries a canonical translation-invariant partial order $\le,$
the \emph{operator order}: for $a,b\in \A_\hermitian,$ the operator order is given by
\begin{equation}
  a \le b\quad\quad :\Longleftrightarrow \quad\quad \skal{\phi}{a(\phi)} \le \skal{\phi}{b(\phi)} \text{ for all $\phi\in \Dom.$}
\end{equation}
This operator order is indeed an antisymmetric relation as a consequence of the polarization identity
$\skal{\phi}{a(\psi)} = \frac{1}{4} \sum_{k=0}^3 \I^{-k} \skal{\phi+\I^k \psi}{a(\phi+\I^k \psi)}$ for $\phi,\psi \in \Dom$
and $a\in \Adbar(\Dom).$ Clearly $\Unit \in \A_\hermitian^+$ and $a^2 b = a^* b a \in \A_\hermitian^+$ for $a\in \A_\hermitian,$ $b\in \A_\hermitian^+$,
so $\A_\hermitian^+$ is stable under multiplication with squares as in Proposition~\ref{proposition:sqrt}.
Moreover, $\A_\hermitian$ is archimedean:
If $a,b\in \A_\hermitian$ fulfil $ka + b \in \A_\hermitian^+$ for all $k\in \NN,$ then
\begin{equation}
  \skal{\phi}{a(\phi)} = \skal{\phi}{a(\phi)} +  \lim_{k\to\infty} \frac{\skal{\phi}{b(\phi)}}{k} = \lim_{k\to\infty}  \frac{\skal{\phi}{(ka+b)(\phi)} }{k} \ge 0
\end{equation}
for all $\phi \in \Dom$ and therefore $a\in \A_\hermitian^+.$ Note that this means that Proposition~\ref{proposition:sqrt}
applies to $\A_\hermitian$ whenever $\A_\hermitian$ is localizable.
Every element of $\A_\hermitian$ is, by definition, a densely defined symmetric (but in general unbounded) operator on the
surrounding Hilbert space $\Hilb.$ It therefore makes sense to ask whether or not an element of $\A_\hermitian$ is (essentially) selfadjoint.

\begin{theorem}
  Let $\Dom$ be a dense linear subspace of a complex Hilbert space $\Hilb$ and $\A$ a commutative $^*$\=/subalgebra of $\Adbar(\Dom).$
  Assume that all elements of $\Unit + \A_\hermitian^+$ are essentially selfadjoint.
  Then $\A_\hermitian$ with the operator order is a localizable archimedean partially ordered commutative ring, and the
  extended Gelfand transformation $\widehat \argument \colon \A_\hermitian \to \AlEv\bigl( \KSpace(\A_\hermitian)\bigr)$
  is a positive ring morphism and an order embedding.
\end{theorem}
\begin{proof}
  Consider $b\in \A_\hermitian \setminus \A_\hermitian^+$ and $a\in \Unit+\A_\hermitian^+,$ then there exists $\psi \in \Dom$
  such that $\skal{\psi}{b(\psi)} < 0.$ Set
  \begin{align*}
    \epsilon &\coloneqq \min\Bigl\{ 1, \frac{1}{2}\bigl( 1 + \norm{\psi} + \norm{b(\psi)} \bigr)^{-1} \abs[\big]{\skal{\psi}{b(\psi)}} \Bigr\} \in {]0,1]}
  \shortintertext{and}
    s &\coloneqq \bigl(\Unit + b^2\bigr)a \in \Unit+\A^+_\hermitian .
  \end{align*}
  As $s$ is essentially selfadjoint by assumption and $\Unit \le s,$ the image of $s$ is dense in $\Hilb$ and in particular
  there exists $\phi \in \Dom$ such that $\norm{(\Unit+b^2)\psi-s(\phi)} \le \epsilon$ (see e.g.~\cite[Prop.~3.9]{schmuedgen:UnboundedSelfAdjointOperatorsOnHilbertSpace}).
  Then
  \begin{align*}
    \norm[\big]{\psi-a(\phi)}^2 + \norm[\big]{b\bigl(\psi-a(\phi)\bigr)}^2
    &=
    \skal[\big]{\psi-a(\phi)}{ \bigl(\Unit+b^2\bigr) \bigl(\psi-a(\phi)\bigr)}
    \\
    &\le
    \norm[\big]{\psi-a(\phi)} \,\norm[\big]{\bigl(\Unit+b^2\bigr) \bigl(\psi-a(\phi)\bigr)}
    \\
    &\le
    \epsilon \norm[\big]{\psi-a(\phi)}
  \end{align*}
  holds. So, firstly, $\norm{\psi-a(\phi)}^2 \le \epsilon \norm{\psi-a(\phi)}$
  shows that $\norm{\psi-a(\phi)} \le \epsilon,$ and then, secondly, $\norm{b(\psi-a(\phi))}^2 \le \epsilon \norm{\psi-a(\phi)} \le \epsilon^2$
  shows that $\norm{b(\psi-a(\phi))} \le \epsilon.$ It follows that
  \begin{align*}
    \abs[\big]{ \skal{\psi}{b(\psi)} - \skal{\phi}{a^2 b(\phi)} }
    &=
    \abs[\big]{ \skal{\psi}{b(\psi)} - \skal{a(\phi)}{ba(\phi)} }
    \\
    &=
    \abs[\Big]{ \skal[\big]{\psi-a(\phi)}{b(\psi)} + \skal[\big]{a(\phi) - \psi }{b\bigl(\psi - a(\phi)\bigr)} + \skal[\big]{\psi}{b\bigl(\psi - a(\phi)\bigr)} }
    \\
    &\le
    \norm{\psi-a(\phi)} \, \norm{b(\psi)} + \norm{a(\phi) - \psi} \, \norm[\big]{ b\bigl(\psi - a(\phi)\bigr) } + \norm{\psi} \, \norm[\big]{ b\bigl(\psi - a(\phi)\bigr) }
    \\
    &\le
    \epsilon \, \bigl( \norm{b(\psi)}  + \epsilon + \norm{\psi}\bigr)
    \\
    &\le
    \frac{1}{2}\abs[\big]{\skal{\psi}{b(\psi)}}
  \end{align*}
  holds. As $\skal{\psi}{b(\psi)} < 0$ this shows that $\skal{\phi}{a^2 b(\phi)} < 0,$ and in particular $a^2 b \notin \A_\hermitian^+.$
  
  We have thus proven that $a^2 \in \Loc(\A_\hermitian)$ for all $a\in \Unit+\A_\hermitian^+.$ This shows that $\A_\hermitian$
  is localizable: Indeed, given any $a\in \A_\hermitian$ and set $b \coloneqq \Unit + (a/2)^2 \in \Unit+\A_\hermitian^+,$
  then $-b \le a \le b$ because $b+a = (\Unit+a/2)^2$ and $b-a = (\Unit-a/2)^2,$ and then $b\le b + (a/2)^2b = b^2 \in \Loc(\A_\hermitian)$
  and $-b^2 \le a \le b^2$ hold. So $\A_\hermitian$ is archimedean, localizable, and by Proposition~\ref{proposition:sqrt}, a partially 
  ordered commutative ring. Therefore the main Theorem~\ref{theorem:main} applies
  and shows that the extended Gelfand transformation $\widehat \argument \colon \A_\hermitian \to \AlEv\bigl( \KSpace(\A_\hermitian)\bigr)$
  is a positive ring morphism and an order embedding.
\end{proof}

% \bibliographystyle{plainurl}
% \bibliography{../../library/bibtex/library.bib}

{
	\normalfont
}

\end{onehalfspace}
\end{document}